\documentclass[12pt,latin9,utf8,reqno]{amsart}

\usepackage{geometry}
\usepackage{amsbsy}
\usepackage{amstext}
\usepackage{amsthm}
\usepackage{amssymb}
\usepackage{graphicx}
\usepackage[hyperfootnotes=false,unicode=true,
bookmarks=false, 
breaklinks=false,pdfborder={0 0 1},backref=section,colorlinks=false]
{hyperref}
\usepackage{subfig}

\makeatletter
\numberwithin{equation}{section}
\numberwithin{figure}{section}
\theoremstyle{plain}
\newtheorem{thm}{\protect\theoremname}[section]
\theoremstyle{plain}
\newtheorem{ass}{Assumption}
\theoremstyle{plain}
\newtheorem{conjecture}[thm]{\protect\conjecturename}
\theoremstyle{remark}
\newtheorem{rem}[thm]{\protect\remarkname}
\theoremstyle{definition}
\newtheorem{defn}[thm]{\protect\definitionname}
\theoremstyle{plain}
\newtheorem{lem}[thm]{\protect\lemmaname}
\theoremstyle{plain}
\newtheorem{prop}[thm]{\protect\propositionname}
\theoremstyle{plain}
\newtheorem{cor}[thm]{\protect\corollaryname}
\theoremstyle{definition}
\newtheorem{definition}{Definition}[section]
\theoremstyle{plain}

\@ifundefined{date}{}{\date{}}

\usepackage{amscd}
\usepackage{comment}
\usepackage{mathrsfs}
\usepackage{setspace}
\usepackage{color}

\DeclareMathOperator{\eig}{Eig}

\DeclareMathOperator{\spec}{spec}
\DeclareMathOperator{\tr}{tr}

\DeclareMathOperator{\adj}{adj}

\providecommand{\conjecturename}{Conjecture}
\providecommand{\corollaryname}{Corollary}
\providecommand{\definitionname}{Definition}
\providecommand{\lemmaname}{Lemma}
\providecommand{\propositionname}{Proposition}
\providecommand{\remarkname}{Remark}
\providecommand{\theoremname}{Theorem}

\makeatother

\providecommand{\conjecturename}{Conjecture}
\providecommand{\corollaryname}{Corollary}
\providecommand{\definitionname}{Definition}
\providecommand{\lemmaname}{Lemma}
\providecommand{\propositionname}{Proposition}
\providecommand{\remarkname}{Remark}
\providecommand{\theoremname}{Theorem}
 \newcommand{\noop}[1]{} 
\begin{document}
	\makeatletter
	\providecommand*{\dd}{\@ifnextchar^{\DIfF}{\DIfF^{}}}
	\def\DIfF^#1{
		\mathop{\mathrm{\mathstrut d}}%
		\nolimits^{#1}\gobblespace}
	\def\gobblespace{
		\futurelet\diffarg\opspace}
	\def\opspace{
		\let\DiffSpace\!
		\ifx\diffarg(
		\let\DiffSpace\relax
		\else
		\ifx\diffarg[
		\let\DiffSpace\relax
		\else
		\ifx\diffarg\{
		\let\DiffSpace\relax
		\fi\fi\fi\DiffSpace}
	
	\global\long\def\bg{\boldsymbol{\gamma}}%
	\global\long\def\bz{\boldsymbol{z}}%
	
	\global\long\def\c{\boldsymbol{c}}%
	\global\long\def\Neum{\mathcal{N}}%
	\global\long\def\bj{\boldsymbol{j}}%

	\global\long\def\at{\left.\right|_{\Omega}}%
	\global\long\def\set#1#2{\left\{  #1\,\,:\,\,#2\right\}  }%
	
	\global\long\def\PG{P_{\Gamma}}%
	\global\long\def\mflat{\Sigma_{\mathrm{flat}}}
	\global\long\def\Gammasym{\Gamma,\text{sym}}%
	\global\long\def\EL{\E_{\text{loops}}}%
	
	\global\long\def\LL{L_{\text{loops}}}%
	
	\global\long\def\TS{\mathcal{T}}%
	\global\long\def\A{\mathcal{A}}%
	\global\long\def\Nvc{\mathcal{N}}%
	\global\long\def\G{\mathcal{G}}%
	
	\global\long\def\prop{\mathcal{P}}%
	
	\global\long\def\L{\mathcal{L}}%
	
	\global\long\def\D{\mathcal{D}}%
	
	\global\long\def\U{\mathcal{U}}%
	
	\global\long\def\Rv{\mathcal{R}^{\left(v\right)}}%
	
	\global\long\def\r{\mathcal{R}}%
	
	\global\long\def\d{\partial}%
	
	\global\long\def\dg{\partial\Gamma}%
	
	\global\long\def\do{\partial\Omega}%
	
	\global\long\def\E{\mathcal{E}}%
	
	\global\long\def\V{\mathcal{V}}%
	
	\global\long\def\Vint{\mathcal{V}\setminus\partial\Gamma}%
	
	\global\long\def\Vin{V_{\textrm{in}}}%
	
	\global\long\def\la{\lambda}%
	
	
	\global\long\def\H{H}%
	
	\global\long\def\Z{\mathbb{Z}}%
	
	\global\long\def\R{\mathbb{R}}%
	
	\global\long\def\C{\mathbb{C}}%
	
	\global\long\def\N{\mathbb{N}}%
	
	\global\long\def\Q{\mathbb{Q}}%
	
	\global\long\def\msing{\Sigma^{\mathrm{sing}}}%
	
	\global\long\def\mreg{\Sigma^{\mathrm{reg}}}%
	
		\global\long\def\zsing{Z_{\Gamma}^{\mathrm{sing}}}%
	
	\global\long\def\zreg{Z_{\Gamma}^{\mathrm{reg}}}%

	\global\long\def\lap{\Delta}%
	
	\global\long\def\na{\nabla}%
	
	\global\long\def\opcl#1{\left[#1\right)}%
	
	\global\long\def\clop#1{\left[#1\right)}%
	
	\global\long\def\bs#1{\boldsymbol{#1}}%
	
	\global\long\def\deg#1{\mathrm{deg}(#1)}%
	\global\long\def\o#1{\mathrm{o}(#1)}%
	\global\long\def\t#1{\mathrm{\tau}(#1)}%
	
	\global\long\def\T{\mathbb{T}}%
	
	\global\long\def\TE{\mathbb{T^{\left|\E\right|}}}%
	
	\global\long\def\BGm{\mu_{\vec{l}}}%
	
	\global\long\def\lv{{\boldsymbol{\ell}}}%
	\global\long\def\lve{\ell_{e}}%
	\global\long\def\lvj{\ell_{j}}%
	\global\long\def\l{\ell}%
	
	\global\long\def\ts{t_{S}}%
	
	\global\long\def\Lv{\vec{L}}%
	
	\global\long\def\av{\vec{\alpha}}%
	
	\global\long\def\kv{{\vec{\kappa}}}%
	
	\global\long\def\xv{\textbf{x}}%
	
	\global\long\def\Tv{\vec{\kappa}}%
	
	\global\long\def\dL{d_{\vec{L}}}%
	
	\global\long\def\sgn{\mathrm{sgn}}%
	
	\global\long\def\undercom#1#2{\underset{_{#2}}{\underbrace{#1}}}%
	\global\long\def\fr#1{\{#1\}_{2\pi}}%

	\global\long\def\diag{\textrm{diag}}%
	
	\global\long\def\as{\boldsymbol{v}}%
	\global\long\def\Vas{\boldsymbol{V}_{\text{a-sym}}}
	\global\long\def\Vs{\boldsymbol{V}_{\text{sym}}}
	\global\long\def\mL{\Sigma_{\text{loops}}}
	\global\long\def\mregL{\mreg_{\text{loops}}}
	
	\global\long\def\ones{{\boldsymbol{1}}}%
	\global\long\def\uone{\textrm{u}(1)}%
	\global\long\def\Tw{\widetilde{\T}^{E}}
	\global\long\def\p{\textrm{p}}%
	
	\title{Generic Laplacian eigenfunctions on metric graphs}
	\author{Lior Alon}
	\address{{\small{}Lior Alon, School of Mathematics, Institute for Advanced Study, Princeton, NJ 08540, USA. e-mail: lalon@ias.edu.il}}
	
	\begin{abstract}
	It is known that up to certain pathologies, a compact metric graph with standard vertex conditions has a Baire-generic set of choices of edge lengths such that all Laplacian eigenvalues are simple and have eigenfunctions that do not vanish at the vertices, \cite{Fri_ijm05,BerLiu_jmaa17}. We provide a new notion of strong genericity, using subanalytic sets, that implies both Baire genericity and full Lebesgue measure. We show that the previous genericity results for metric graphs are strongly generic. In addition, we show that generically the derivative of an eigenfunction does not vanish at the vertices either. In fact, we show that generically an eigenfunction fails to satisfy any additional vertex condition. Finally, we show that any two different metric graphs with the same edge lengths do not share any non-zero eigenvalue, for a generic choice of lengths, except for a few explicit cases where the graphs have a common edge-reflection symmetry. The paper concludes by addressing three open conjectures for metric graphs that can benefit from the tools introduced in this paper.  
	\end{abstract}
	
	\maketitle
	\section{Introduction}A metric graph $ (\Gamma,\lv) $ is a finite graph $ \Gamma $ of $ N $ edges and a positive vector $ \lv=(\ell_{1},\ldots,\ell_{N}) $, assigning a positive length $ \ell_{j} $ to every edge $ e_{j} $. Each edge $ e_{j} $ is identified with the interval $ [0,\ell_{j}] $. That is, we fix an (arbitrary) orientation on $ e_{j} $ and then parameterize $ e_{j} $ by arc-length parameter $x_{j}\in [0,\ell_{j}] $. In this way, $ (\Gamma,\lv) $ is a one-dimensional Riemannian manifold with singularities equipped with a uniform metric $ dx_{j} $ on each edge $ e_{j} $. Given a function $f: (\Gamma,\lv)\to\C $ we consider its restriction to each edge $ e_{j} $ as a function on the interval, $ f|_{e_{j}}:[0,\ell_{j}]\to\C $. The (one-dimensional) Laplacian $ \Delta $ acts edgewise, 
	\[(\Delta f)|_{e_{j}}(x_{j}):=-\frac{d^{2}}{d_{x_{j}}^{2}}f|_{e_{j}}(x_{j}).\] 
	We require the functions to satisfy standard matching conditions on the vertices (see Definition \ref{def: standard vertex conditions}), in which case the Laplacian has a discrete non-negative spectrum and each eigenspace is spanned by real eigenfunctions. We refer to the eigenvalues and eigenfunctions of the Laplacian as the eigenvalues and eigenfunctions of $ (\Gamma,\lv) $. Consider
	\[\spec(\Gamma,\lv):=\set{k\in\R_{\ge0}}{k^2 \mbox{   is an eigenvalue of   }(\Gamma,\lv)},\]
	which we treat as a multiset where each eigenvalue is repeated according to its multiplicity.
	In this paper we analyze the eigenvalues and eigenfunctions of $ (\Gamma,\lv) $, while $ \Gamma $ is fixed and $ \lv $ varies along $ \R_{+}^{N} $, and point out some properties that hold generically in $ \lv $.

	Generic spectral properties for metric graphs were first studied by Friedlander in 2005 \cite{Fri_ijm05}. Motivated by the Sturm--Liouville theory on intervals and the genericity works of Albert \cite{Albert_thesis72} and Uhlenbeck \cite{Uhl_bams72} on compact manifolds, Friedlander showed that with one exception,\footnote{The exceptional graphs are polygons,i.e., the circle $ S^1 $ with finitely many degree-two vertices, which are removable singularities.} for any graph $ \Gamma $ of $ N $ edges and a generic $ \lv\in\R_{+}^N $, all eigenvalues of $ (\Gamma,\lv) $ are simple. The next result was due to Berkolaiko and Liu \cite{BerLiu_jmaa17} in 2017. Their motivation was a series of works on zeros of eigenfunctions, where the repeated assumption is that the eigenvalue is simple and the eigenfunction does not vanish at any vertex; see \cite{GnuSmiWeb_wrm04, Ber_cmp08} and the references therein. Berkolaiko and Liu \cite{BerLiu_jmaa17} showed that for any graph $ \Gamma $ that has no loops,\footnote{If we neglect degree-two vertices (which are removable singularities), then a loop is an edge connecting a vertex to itself. } and for a generic $ \lv $, any eigenfunction of $ (\Gamma,\lv) $ would not vanish at any vertex. Let us remark that the result of \cite{BerLiu_jmaa17} also treats graphs with loops and other vertex conditions. In both \cite{Fri_ijm05} and  \cite{BerLiu_jmaa17}, the term ``generic $ \lv $" means that $ \lv $ belongs to a set $ G\subset \R_{+}^{N} $ that is \emph{Baire generic}; namely, $ G $ contains a countable intersection of open dense sets. A Baire-generic set is generic in a topological sense but may be very small in terms of measure theory. In fact, a Baire-generic set can have a zero Lebesgue measure.\footnote{For example, number the rational points in $ \R_{+}^{N} $ and let $ B_{j,\epsilon} $ be the open ball of volume $ \epsilon 2^{-j} $ centered at the $ j $-th rational point. The set $ O_{\epsilon}=\bigcup_{j\in\N} B_{j,\epsilon} $ is open dense and has Lebesgue measure at most $ \epsilon $. Then $ G=\bigcap_{n\in\N} O_{\frac{1}{n}} $ is Baire generic and has measure zero. By contrast, its complement $ G^{c} $ is a full measure set which is not Baire generic.} If, for example, we choose $ \lv\in (0,T)^{N} $ uniformly at random for some $ T\gg 1 $, one may ask:
	\begin{center}
		\textbf{Does $(\Gamma,\lv) $ satisfy these generic properties with high probability?}
	\end{center}
	We would like to say that the ``good" set $ G $ is both Baire generic and has full Lebesgue measure, in which case the above question can be answered in the affirmative. To this end, we define a new notion of \emph{strong genericity}, which classifies $ G $ in terms of \emph{subanalytic sets} (see \cite{Gab968projections,BieMil1988semianalytic} or definition \ref{def: subanalytic}). Heuristically, a subanalytic set is a set that locally can be defined as a projection of a level set (or sub-level set) of a real analytic function.  
	\begin{defn}[Strong genericity]
		We say that $ G\subset \R_{+}^{N} $ is \emph{strongly generic} if its complement in $ \R_{+}^{N} $ is a countable union of closed subanalytic sets of positive codimension. 
	\end{defn}
	We will show later that a strongly generic $ G $ is both Baire generic and has full Lebesgue measure. We believe that strong genericity should appear in various eigenvalue problems for operators that depends analytically on finitely many parameters. For example, generic sets of similar nature appear in genericity results for Laplacian eigenvalues of planar polygons and hyperbolic triangles \cite{JudHil2009generic,HilJud2018hyperbolic}.

	 Another type of genericity result is of an ergodic nature. We say that $ \lv $ is \emph{$ \Q $-independent} if $ \lv\cdot\textbf{q}\ne0 $ for every non-zero rational vector $ \textbf{q}\in\Q^{N}\setminus\{0\} $. The behavior of the spectrum of $ (\Gamma,\lv) $ when $ \lv $ is $ \Q $-independent is believed to be ``almost chaotic" \cite{KotSmi_prl97,BarGas_jsp00}. 
	 \begin{defn}[Ergodic genericity]
	 	Let $ \prop $ be some property of the eigenpairs of $ (\Gamma,\lv) $. Define the set $ \spec(\Gamma,\lv,\prop) $ of square root eigenvalues $ k\in \spec(\Gamma,\lv) $ with a corresponding eigenfunction $ f $ such that $ (k^2,f) $ satisfies $ \prop $. We say that $ \prop $ is \emph{ergodically generic}, if, for any $ \Q $-independent $ \lv $,
	 	\[\lim_{T\to\infty}\frac{|\spec(\Gamma,\lv,\prop)\cap[0,T]}{|\spec(\Gamma,\lv)\cap[0,T]|}=1.\] 
	 	Namely, for \textbf{any} $\Q $-independent $ \lv $, \textbf{almost every} eigenpair of $ (\Gamma,\lv) $ satisfies $ \prop $.  
	 \end{defn}
 The Baire generic results of \cite{Fri_ijm05,BerLiu_jmaa17} were shown to be ergodically generic (without using this term) in \cite{AloBanBer_cmp18}. As the term ``ergodic" suggests, there is an ergodic system in the background. Let $ \T^{N} $ be the subset of $ \bz\in\C^{N} $ with $ |z_{j}|=1 $ for all $ j $. The \emph{secular manifold} of a graph $ \Gamma $ with $ N $ edges is a hypersurface $ \Sigma(\Gamma)\subset\T^{N} $ such that for any $ \lv\in\R_{+}^{N} $ and $ k\ge 0 $,
	 \[\exp(ik\lv):=(e^{ik\ell_{1}},e^{ik\ell_{2}},\ldots,e^{ik\ell_{N}})\in\Sigma(\Gamma)\iff k\in\spec(\Gamma,\lv).\]
	  If $ \lv $ is $\Q $-independent, the sequence of $ \exp(ik\lv) $ for $ k\in\spec(\Gamma,\lv) $ equidistributes with respect to a certain measure on $ \Sigma(\Gamma) $ \cite{BarGas_jsp00,BerWin_tams10,CdV_ahp15}. In this ergodic case, certain averages over $ \spec(\Gamma,\lv) $ can replaced by integration over $ \Sigma(\Gamma) $ \cite{AloBanBer_cmp18, BarGas_jsp00, GnuKeaPio_ap10,GnuSmi_ap06}.

	In this paper we extend the previous genericity results for metric graphs. Consider the properties discussed so far: 
	 \begin{enumerate}
		\item the eigenvalue $ k^2 $ is simple, 
		\item the eigenfunction $ f $ does not vanish at any vertex,
	\end{enumerate}
and the following additional property,
\begin{enumerate}
	\item[(3)] the derivative of $ f $ does not vanish at vertices of degree larger than one.\footnote{On vertices of degree one the derivative vanish due to the vertex condition. }   
\end{enumerate} 
In Theorem \ref{thm: non vanishing}, we show that properties (1), (2), and (3) are strongly and ergodically generic, if we assume that the graph has no loops. The case of graphs with loops is treated later by excluding eigenfunctions that are supported on a single loop.  

 At this point, we may observe that properties (2) and (3) can be interpreted as additional (scaling-independent) vertex conditions. Intuitively, adding another vertex condition should make the system over-determined and so we do not expect to see, generically, eigenfunctions that satisfy the additional condition. However, certain conditions, such as $ f(v)=f(u) $ for different vertices, may appear infinitely often in the presence of $ \lv $-independent reflection symmetries. Excluding such symmetries leads to a dichotomy, as Theorem \ref{thm: polynomial vertex conditions} suggests. Every scaling-independent condition is either satisfied for all eigenfunctions of all simple eigenvalues for every $ \lv $, or it is (both strongly and ergodically) generically never satisfied. This dichotomy extends to polynomial scaling-invariant vertex conditions. 
	 
	 Next, we go back to a well known theorem on the spectrum of metric graphs, whereby $ (\Gamma,\lv) $ can be constructed from the spectrum, as long as $ \lv $ is $ \Q $-independent \cite{GutSmi_jpa01,KurNow_jpa05}. In particular, two different metric graphs have different spectra $ \spec(\Gamma,\lv)\ne \spec(\Gamma',\lv')$ under the assumption that $ \lv $ and $ \lv' $ are $ \Q $-independent. However, a priori $ \spec(\Gamma,\lv)$ and $\spec(\Gamma',\lv') $ may only disagree for a small set of eigenvalues. In Theorem \ref{thm: disjoint spectrum}, we compare the spectra of distinct graphs $ 
	 \Gamma $ and $ \Gamma' $ with equal edge lengths $ \lv=\lv' $. We show that except for some pathological cases, for any two distinct graphs of $ N $ edges, there is a strongly generic set of $ \lv $'s for which 
	 \[\spec(\Gamma,\lv)\cap(\Gamma',\lv)=\{0\}.\]
	 Moreover, for any $ \Q $-independent $ \lv $,
	 	\[\lim_{T\to\infty}\frac{|\spec(\Gamma,\lv)\cap(\Gamma',\lv)\cap[0,T]}{|\spec(\Gamma,\lv)\cap[0,T]|}=1.\]

The general strategy in the genericity proofs of all the theorems mentioned above is similar. Consider the torus subset 
\[\Sigma(\Gamma):=\set{\exp(ik\lv)\in\T^{N}}{k\in\spec(\Gamma,\lv)},\qquad\exp(ik\lv):=(e^{ik\ell_{1}},e^{ik\ell_{2}},\ldots,e^{ik\ell_{N}}).\]
A main fact being used (see Lemma \ref{lem: main genericty lem}), is that whenever $ B\subset\Sigma(\Gamma) $ is a subvariety\footnote{By ``subvariety of $ \Sigma(\Gamma) $" we mean the intersection of $ \Sigma(\Gamma)\subset\C^{N} $ with a common zero set of finitely many polynomials in $ \C^{N} $.} of $ \Sigma(\Gamma) $ with positive codimension, then it is strongly and ergodically generic to have $\exp(ik\lv)\notin B  $. Hence, to prove that a certain property $ \prop $ is generic, we first need to construct a ``bad" subvariety $ B $ that captures the negation of $ \prop $. We then need to show that $ B $ has a positive codimension in $ \Sigma(\Gamma) $. Let us elaborate on these two steps.\\

First step: Constructing a subvariety $ B $ for a given property $ \prop $. For the multiplicity of eigenvalues, $ B $ will be the singular set of $ \Sigma(\Gamma) $. To capture properties of eigenfunctions we introduce the trace space. To every eigenpair $ (f,k^2) $ of $ (\Gamma,\lv) $ we associate a vector $ \tr_{k}(f)\in\C^{4N} $, called the \emph{scale-invariant trace}, which consists of a 4-tuple $ (A_{j},B_{j},C_{j},D_{j}) $ for every edge $ j $. These are the amplitudes for which the restriction of $ f $ to $ e_{j} $ is
\[f|_{e_{j}}(t_{j})=A_{j}\cos(kt_{j})+B_{j}\sin(kt_{j})=C_{j}\cos(k(\ell_{j}-t_{j}))+D_{j}\sin(k(\ell_{j}-t_{j})).\]
The collection of $ A $'s and $ C $'s is often called the Dirichlet trace and the Neumann trace is the collection of $ B $'s and $ D $'s, scaled by $ k $. The \emph{trace space} is defined as
	\[\TS(\Gamma):=\set{(\exp(ik\lv),\tr_{k}(f))\in\Sigma(\Gamma)\times\C^{4N}}{(k^2,f) \mbox{  is an eigenpair of }(\Gamma,\lv)}.\]
Denoting the points in $ \TS(\Gamma) $ by $ (\bz,\xv)=(\exp(ik\lv),\tr_{k}(f)) $, we show that the $ \xv $ fiber above any regular point $ \bz\in\Sigma(\Gamma) $ is a one-dimensional complex vector space (spanned by a real vector), and we analyze its $ \bz $ dependence. Properties of eigenfunctions are then carried over to properties of $ \xv $ fibers, which can be projected down to $\bz \in\Sigma(\Gamma) $. This general procedure associates a ``bad" subvariety $ B $ to a given property $ \prop $.\\

Second step: Showing that $ B $ has a positive codimension. To this end, we use the irreducible structure of $ \Sigma(\Gamma) $, which was conjectured by Colin de Verdière \cite{CdV_ahp15} and recently proved by Kurasov and Sarnak \cite{KurasovBook,KurSar2022}. Neglecting some pathologies for the moment, this result provides the needed dichotomy; any $ B $ subvariety of $ \Sigma(\Gamma) $ is either equal to $ \Sigma(\Gamma) $ or it has a positive codimension in $ \Sigma(\Gamma) $.\\

The structure of the paper is as follows. Section 2 provides some necessary preliminaries to Section 3, in which the main results are presented. In Section 4 we construct and analyze the secular manifold $ \Sigma(\Gamma) $ and the trace space $ \TS(\Gamma) $. Section 5 deals with the irreducible structure of $ \Sigma(\Gamma) $ as shown by Kurasov and Sarnak \cite{KurasovBook,KurSar2022}. In Section 6 we prove the main results. In Section 7 we suggest future work. In particular, three open conjectures regarding metric graphs that may benefit from the trace space and genericity concepts are introduced.    
\subsection*{Acknowledgments}
The author would like to thank Peter Sarnak, Mark Goresky, Karen Uhlenbeck, Pavel Kurasov, and Ram Band for insightful discussions, important remarks and relevant references.    
The author was supported
by the Ambrose Monell Foundation and the Institute for Advanced Study.

	 
	
	

	\section{Preliminaries and notation}
	\subsection{Metric graphs}
	Let $ (\Gamma,\lv) $ be a metric graph with $ N $ edges. It is convenient to describe functions on $ (\Gamma,\lv) $ in terms of their restrictions to edges. Using the arc-length parameter $ t_{j}\in[0,\ell_{j}] $ along each edge $ e_{j} $, we can write the restrictions of a function $ f:(\Gamma,\lv)\to\C $ to edges as univariate functions, i.e.,
		\[f|_{e_{j}}:[0,\ell_{j}]\to\C,\qquad j=1,2,\ldots,N.\]
	In this way, we can associate an $ L^2 $ Hilbert space and an $ H^2=W^{2,2} $ Sobolev space to each edge to get
	\[L^2(\Gamma,\lv):=\oplus_{j=1}^{N} L^2([0,\ell_{j}]),\quad\mbox{and}\quad H^2(\Gamma,\lv):=\oplus_{j=1}^{N} H^2([0,\ell_{j}]). \]
	It is a standard result that functions in $ H^2([0,\ell_{j}]) $ are $ C^1 $ (i.e., have a continuous derivative). Given $ f\in H^2(\Gamma,\lv) $ we denote its derivative along each edge by $ f'|_{e_{j}} $. Its second derivative $ f''|_{e_{j}} $ is defined as a weak derivative. The non-negative Laplacian $ \Delta: H^2(\Gamma,\lv)\to L^2(\Gamma,\lv)$ acts edgewise by
	\[(\Delta f)|_{e_{j}}=-f''|_{e_{j}},\qquad j=1,2,\ldots, N.\]   
	To get a self-adjoin operator we need to specify a choice of \emph{vertex conditions} (in analogy with boundary conditions). The Dirichlet and Neumann traces of $ f|_{e_{j}} $ are defined by,
	\begin{align*}
		\gamma_{D}(f|_{e_{j}}):= & (f|_{e_{j}}(0),f|_{e_{j}}(\ell_{j}))\in\C^{2},\quad\mbox{and}\\
		\gamma_{N}(f|_{e_{j}}):= & (f'|_{e_{j}}(0),-f'|_{e_{j}}(\ell_{j}))\in T_{0}\C\times T_{\ell_{j}}\C=\C^{2}.
	\end{align*}
	The sign in the last derivative reflects that it is a normal (or outgoing) derivative. Given $ f\in H^2(\Gamma,\lv) $, its Dirichlet trace $ \gamma_{D}(f)\in\C^{2N} $ is the collection of the Dirichlet traces $ \gamma_{D}(f|_{e_{j}}) $ for all edges. The Neumann trace $ \gamma_{N}(f)\in\C^{2N} $ is defined in the same manner. Given $ f,g\in H^2(\Gamma,\lv)$ we calculate, using integration by parts, 
	\[\langle\Delta f,g\rangle_{L^2(\Gamma,\lv)}-\langle f,\Delta g\rangle_{L^2(\Gamma,\lv)}=\langle\gamma_{N}(f),\gamma_{D}(g)\rangle_{\C^{2N}}-\langle\gamma_{D}(f),\gamma_{N}(g)\rangle_{\C^{2N}}.\]
	Hence, $ \Delta $ is self-adjoint when restricted to a dense domain in $ H^2(\Gamma,\lv) $ on which the above right-hand side vanish. To this end, we impose \emph{vertex conditions}; a restriction of the traces $ (\gamma_{D}(f),\gamma_{N}(f)) $ to a $ 2N $-dimensional subspace of $ \C^{4N} $ on which the sesquilinear form on the right-hand side vanishes.\footnote{If we restrict to $ \R^{4N} $ instead of $ \C^{N} $, then this bi-linear form is the standard symplectic form on $ \R^{4N} $ and a $ 2N $-dimensional subspace on which it vanishes is called \emph{Lagrangian}; See \cite{BerLatSuk2019limits}.} We only consider the \emph{standard vertex conditions} (also known as Neumann or Kirchhoff). 
\begin{defn}[Standard Vertex Conditions]\label{def: standard vertex conditions}
	Let $ v $ be a vertex and let $ \E_{v,o} $ be the set of edges whose origin is $ v $, such that $ t_{j}=0 $ at $ v $. Let $ \E_{v,t} $ be the edges that terminate at $ v $, such that $ t_{j}=\ell_{j} $ at $ v $. The standard vertex conditions at $ v $ are
	\begin{align*}
		\mbox{Continuity : }\qquad& \forall e_{j}\in \E_{v,o},\forall e_{i}\in\E_{v,t},\quad f|_{e_{j}}(0)=f|_{e_{i}}(\ell_{i})=:f(v),\quad\mbox{and}\\
		\mbox{Balanced derivatives : }\qquad& \sum_{e_{j}\in\E_{v,o}}f'|_{e_{j}}(0)+\sum_{e_{i}\in\E_{v,t}}(-f'|_{e_{i}}(\ell_{i}))=0.
	\end{align*}
We define $ \mathcal{D}_{\mathrm{standard}}(\Gamma,\lv)\subset H^2(\Gamma,\lv) $ to be the subspace of functions that satisfy the standard vertex conditions at every vertex. 
\end{defn}
The Laplacian, restricted to $ \mathcal{D}_{\mathrm{standard}}(\Gamma,\lv) $ is self-adjoint, non-negative, and has a discrete\footnote{Here we assume that the graph is finite.} spectrum with eigenvalues of finite multiplicity and real eigenfunctions; see \cite{BerKuc_graphs, GnuSmi_ap06} for a thorough review of the subject. Also, zero is a simple eigenvalue whenever $ \Gamma $ is connected. From here on, we focus on the eigenvalue problem
\[\Delta f =k^2 f,\qquad f\in \mathcal{D}_{\mathrm{standard}}(\Gamma,\lv).\] 
We refer to solutions $ (k^2,f)\in\R_{\ge0}\times \mathcal{D}_{\mathrm{std}}(\Gamma,\lv)$ as eigenpairs of $ (\Gamma,\lv) $, where $ k^2 $ is an eigenvalue and $ f $ an eigenfunction $ (\Gamma,\lv) $. By a common abuse of terminology, we refer to the set of non-negative square roots of eigenvalues as the spectrum, 
\[\spec(\Gamma,\lv)=\set{k\in\R_{\ge0}}{k^2\mbox{  is an eigenvalue of  } (\Gamma,\lv)},\]
that should be understood as a multi-set where each $ k $ is repeated according to its multiplicity. The multiplicity of $ k $ is the dimension of the associated eigenspace 
\[\eig(\Gamma,\lv,k):=\set{f\in \mathcal{D}_{\mathrm{standard}}(\Gamma,\lv)}{\Delta f =k^2 f}.\]
An eigenvalue is called \emph{simple} when $ \dim(\eig(\Gamma,\lv,k))=1 $, and \emph{multiple} when $ >1 $.\\

If we scale a graph $ (\Gamma,\lv)\mapsto (\Gamma,r\lv)$ by some positive factor $ r>0 $, the functions in $ \mathcal{D}_{\mathrm{standard}}(\Gamma,\lv) $ are mapped to $ \mathcal{D}_{\mathrm{standard}}(\Gamma,r\lv) $ by  $ f\mapsto r.f $ with $ (r.f)|_{e_{j}}(rt_{j})=f|_{e_{j}}(r_{j}) $. The spectrum and the traces are scaled as follows:
\begin{align*}
	\spec(\Gamma,r\lv)= & \set{\frac{k}{r}}{k\in\spec(\Gamma,\lv)},\\
	(\gamma_{D}(r.f),\gamma_{N}(r.f))=& (\gamma_{D}(f),\frac{1}{r}\gamma_{N}(f)).
\end{align*}
\begin{defn}\label{def: scale invaraint trace}
	Given an eigenpair $ (k^2,f) $ of $ (\Gamma,\lv) $, define its \emph{scale-invariant trace} by
	\[\tr_{k}(f):=(\gamma_{D}(f),\frac{1}{k}\gamma_{N}).\]
	In the case $ k=0 $, the only eigenfunctions are constant\footnote{We assume that the graph is connected.} and we define $ \frac{1}{k}\gamma_{N}=0 $. Unless stated otherwise, from now on, when referring to the trace of a function, we will always consider the scale-invariant trace.
\end{defn}
\begin{rem}\label{rem: bijection}
	The definition of $ \tr_{k}(f) $ in the Introduction agrees with Definition \ref{def: scale invaraint trace}, up to a reordering of the coordinates (change of basis). In the Introduction, $ \tr_{k}(f) $ is defined as the collection of 4-tuples $ (A_{j},B_{j},C_{j},D_{j}) $ per edge $ e_{j} $, such that the restriction of $ f $ to $ e_{j} $ is given by
	\begin{align*}
		f|_{e_{j}}(t_{j})= & A_{j}\cos(kt_{j})+B_{j}\sin(kt_{j})\\
		= & C_{j}\cos(k(\ell_{j}-t_{j}))+D_{j}\sin(k(\ell_{j}-t_{j})).
	\end{align*}
\end{rem}
\begin{rem}
	The fact that the standard vertex conditions are decoupled into equations on $ \gamma_{D} $ and on $ \gamma_{N} $ separately makes them scaling-invariant. This means that $ (\gamma_{D}(f),\gamma_{N}(f)) $ satisfies the vertex conditions if and only if $ \tr_{k}(f) $ does. 
\end{rem}

We may consider the trace as a map $ \tr_{k}:\mathcal{D}_{\mathrm{standard}}\to\C^{4N} $. This is a linear map and its restriction to an eigenspace $ \eig(\Gamma,\lv,k) $ is injective with an explicit inverse, as can be seen in Remark \ref{rem: bijection}. A solution $ (\lv,k,f)$ such that $ (k^2,f) $ is an eigenpair of $ (\Gamma,\lv) $ can be parameterized by the vector $ (\lv,k,\tr_{k}(f)) $. The space of solutions associated to a graph $ \Gamma $, ranging over all possible $ \lv\in\R_{+}^{N} $, may be parameterized as 
\[X(\Gamma):=\set{(\lv,k,\tr_{k}(f))\in\R_{+}^{N}\times\R_{\ge 0}\times \C^{4N}}{k\in\spec(\Gamma,\lv),~~f\in\eig(\Gamma,\lv,k)}.\]
The space of solutions $ X(\Gamma) $ has the following symmetry, denoting $ \exp(ik\lv):=(e^{ik\ell_{1}},\ldots,e^{ik\ell_{N}}) $,
\[ \exp(ik\lv)=\exp(ik'\lv')\quad\Rightarrow\quad\tr_{k}(\eig(\Gamma,\lv,k))= \tr_{k}(\eig(\Gamma,\lv',k')),\]
which can be derived from the explicit expression $ f $ in terms of $ \tr_{k}(f) $ in Remark \ref{rem: bijection}. We define the \emph{trace space} of $ \Gamma $ as the quotient of $ X(\Gamma) $ by this symmetry.
\begin{defn}
	Given a graph $ \Gamma $ of $ N $ edges, we define the \emph{trace space} of $ \Gamma $ by
	\[\TS(\Gamma):=\set{(\exp(ik\lv),\tr_{k}(f))\in\T^{N}\times \C^{4N}}{k\in\spec(\Gamma,\lv),~~f\in\eig(\Gamma,\lv,k^2)}.\]
	The \emph{secular manifold} (a manifold with singularities) is the projection of $ \TS(\Gamma) $ onto the $ \T^{N} $ coordinates, i.e.,
	\[\Sigma(\Gamma):=\set{\exp(ik\lv)\in\T^{N}}{k\in\spec(\Gamma,\lv)}.\]
\end{defn}
\begin{rem}
	It is common to work with $ \R^{N}/2\pi\Z^{N} $ instead of $\T^{N}  $, in which case the term ``secular manifold" refers to $ \set{\xv\in\R^{N}/2\pi\Z^{N}}{\exp(i\xv)\in\Sigma(\Gamma)} $. 
\end{rem}
\subsection{Semianalytic and subanalytic sets} We introduce the notions of \emph{semianalytic} and \emph{subanalytic} sets, as introduced in the works of Gabrielov \cite{Gab968projections} and Bierstone and Milman \cite{BieMil1988semianalytic}. Heuristically, these are sets in $ \R^{N} $ that can be constructed from equalities and inequalities of real analytic functions. The motivating example is that in 1D, zero sets of real analytic functions have no accumulation points. In higher dimensions this is not the case but we still expect sets of this form to have a ``rigid" structure and ``good" properties. In particular, such sets have an integer-valued Hausdorff dimension, and exhibit no fractal or Cantor set behavior. We note that in a more modern language, the class of subanalytic sets is a special case of \emph{O-minimal  theory}, but we will not make use of this fact. 
\begin{definition}[Semianalytic Set]\label{def: analytic set}
		A set $ X\subset\R^{N} $ is called \emph{semianalytic} if any point $ \xv\in X $ has a neighborhood $ \Omega_{\xv}\subset\R^{N} $ such that $ X\cap\Omega_{\xv} $ is a finite union and intersection of subsets of the form 
		\[\set{\xv'\in\Omega_{\xv}}{f_{i}(\xv')=0,~~g_{j}(\xv')<0,~~~~i=1,2,\ldots,i_{0},~~~j=1,2,\ldots,j_{0}},\]
		where the functions $ f_{i} $ and $ g_{j} $ are real analytic on $ \Omega_{\xv} $ for all $ i $ and $ j $. Notice that having $ g_{j}\le 0 $ instead of $ g_{j}<0 $ is allowed due to the finite union.
	\end{definition}
Recall that a relatively compact set, is a set whose closure is compact. We may now define subanalytic sets. 
	\begin{definition}[Subanalytic Set]\label{def: subanalytic}
		A \emph{subanalytic} set $ X\subset\R^{N} $ is such that any point $ \xv\in X $ has a compact neighborhood $ \Omega_{\xv}\subset\R^{N} $ such that $ X\cap\Omega_{\xv}$ is the projection of a relatively compact semianalytic set. Namely, $ X\cap\Omega_{\xv}$ is a finite union and intersection of sets of the form 
	\[\set{\xv'\in\Omega_{\xv}}{\exists  \textbf{y}\in I ~~\mbox{  s.t.  } f_{i}(\xv',\textbf{y})=0,~~g_{j}(\xv',\textbf{y})<0,~~~~i=1,2,\ldots,i_{0},~~~j=1,2,\ldots,j_{0}},\]
	where $ I $ is a compact subset of $ \R^{m} $ for some $ m\in\N $. 
\end{definition}
It is not hard to verify that a subanalytic set $ X $ is closed if all the local inequalities defining it are non-strict, i.e., $ g_{j}\le0 $ instead of $ g_{j}<0 $.
	\section{Main results}
 Given a graph $ \Gamma $ and some property of eigenpairs denoted by $ \prop $, let $ G $ be the set of $ \lv\in\R_{+}^{N} $ for which every eigenpair of $ (\Gamma,\lv) $ satisfies $ \prop $. Here and throughout, the term ``every eigenpair" excludes the eigenpair with $ k=0 $ and constant $ f $. Let $ \spec(\Gamma,\lv,\prop) $ be the subset of $\spec(\Gamma,\lv) $ of $ k $ values for which there is an eigenpair $ (k^2,f) $ satisfying $ \prop $. 
 \begin{definition}[Strong and Ergodic Genericity]\label{def: strong and ergodic genericity}
 	We say that the property $ \mathcal{P} $ is
 	\begin{enumerate}
 		\item \emph{Strongly generic}: if the complement $ G^{c}=\R_{+}^{N}\setminus G $ is a countable union of closed subanalytic sets of positive codimension in $ \R_{+}^{N} $. We also call $ G $ a strongly generic subset of $ \R_{+}^{N} $ in such a case. 
 		\item \emph{Ergodically generic}: if for any $ \Q $-independent $ \lv\in \R_{+}^{N} $,
 		\[\lim_{T\to\infty}\frac{|\spec(\Gamma,\lv,\prop)|}{|\spec(\Gamma,\lv)|}=1.\]
 	\end{enumerate} 
 \end{definition}
\begin{rem}[Strong genericity implies Baire genericity and full measure]
	The reason that strongly generic $ G $ is also Baire generic and has full measure is as follows. A subanalytic set has a well-defined integer-valued Hausdorff dimension, and therefore a closed subanalytic set of positive codimension in $ \R^{N} $ is a closed nowhere dense set of zero Lebesgue measure. Hence, a countable union of closed subanalytic sets of positive codimension has a complement that is both Baire generic and of full measure.
\end{rem}
\begin{ass}\label{ass: assumptions}
	The graph $ \Gamma $ is finite, connected, has no vertices of degree two (removable singularities), and is not a loop graph (i.e., has a vertex of degree different than two).
\end{ass}
Under Assumption \ref{ass: assumptions}, a \emph{loop} is an edge connecting a vertex to itself (not to be confused with a simple closed path). We say that a trace $ \tr_{k}(f) $ is \emph{non vanishing} if it has no zeros, except for entries corresponding to derivatives at vertices of degree one.
\begin{figure}
	\includegraphics[height=0.2\paperheight]{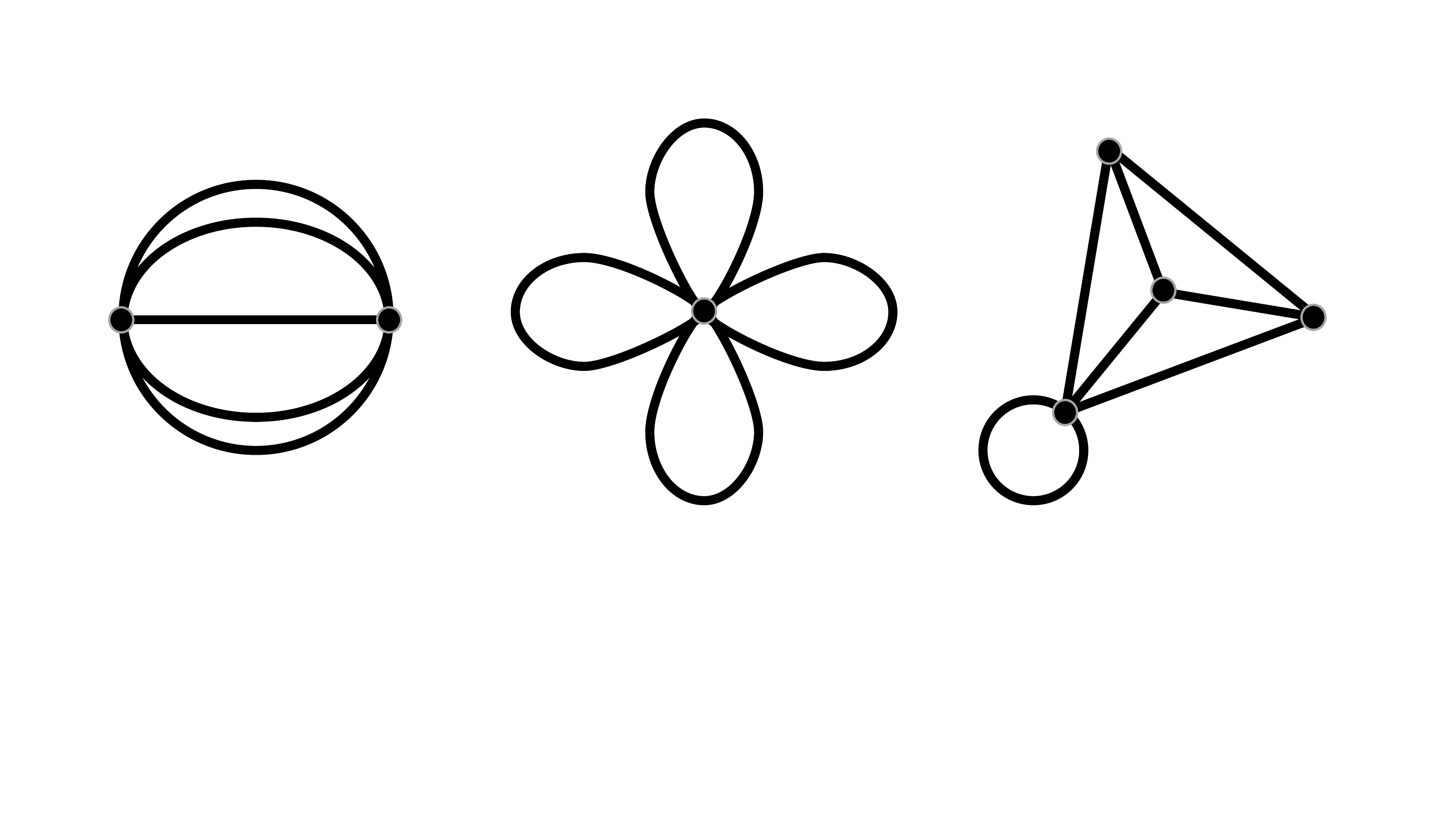}
	
	\caption{\label{fig: Graphs figure} Left: A Mandarin graph with 5 edges. Middle: A flower graph with 4 loops. Right: A graph with one loop.}
\end{figure}
\begin{thm}[Non-vanishing trace]\label{thm: non vanishing}
	If $ \Gamma $ is a graph satisfying Assumption \ref{ass: assumptions}, then the following properties of eigenpairs are both strongly and ergodically generic.
	\begin{enumerate}
		\item $ k^2 $ is a simple eigenvalue,  and
		\item Either $ \tr_{k}(f) $ is non-vanishing, or $ f $ is supported on a single loop (if such exists).
	\end{enumerate}
 In particular, if $ \Gamma $ has no loops, then generically $ \tr_{k}(f) $ is non-vanishing.
\end{thm}
In Theorem \ref{thm: non vanishing} we distinguish graphs with loops from other graphs. For the next theorem we will also introduce two special types of graphs. A \emph{mandarin}\footnote{Mandarin graphs are sometimes referred to as pumpkin or watermelon graphs.} graph has only two vertices and each edge connects the two. A \emph{flower} graph has only one vertex, and every edge is a loop. See Figure \ref{fig: Graphs figure} for examples. 
\begin{thm}[No common spectrum]\label{thm: disjoint spectrum} Except for two cases, given any pair of non-isomorphic graphs $ \Gamma $ and $ \Gamma' $, both have $ N $ edges and satisfy Assumption \ref{ass: assumptions}:   
	\begin{enumerate}
		\item There is a strongly generic set $ G\subset\R_{+}^{N} $ such that for any $ \lv=\lv'\in G $,
		\[\spec(\Gamma,\lv)\cap\spec(\Gamma',\lv)=\{0\}.\]
		\item For any $ \Q $-independent $ \lv $, the joint spectrum has density zero in $ \spec(\Gamma,\lv) $,i.e.,
		\[\lim_{T\to\infty}\frac{|\spec(\Gamma,\lv)\cap\spec(\Gamma',\lv)\cap[0,T]|}{|\spec(\Gamma,\lv)\cap[0,T]|}=0.\]
	\end{enumerate}
The two exceptional cases are,
\begin{enumerate}
	\item[i)] If $ \Gamma $ and $ \Gamma' $ share a common loop, say $ e_{j} $, then for any choice of $ \lv=\lv'\in\R_{+}^{N} $,
	\[\set{k=\frac{2\pi}{\ell_{j}}n}{n\in\N}\subset\spec(\Gamma,\lv)\cap\spec(\Gamma',\lv),\]
	which means that the common spectrum has positive density,
		\[\liminf_{T\to\infty}\frac{|\spec(\Gamma,\lv)\cap\spec(\Gamma',\lv)\cap[0,T]|}{|\spec(\Gamma,\lv)\cap[0,T]|}\ge\frac{2L}{\ell_{j}},\qquad  L=\sum_{j=1}^{N}\ell_{j}.\]
	\item [ii)] If $ \Gamma $ is a mandarin graph and $ \Gamma' $ is a flower (or vice versa), then for any choice of $ \lv=\lv'\in\R_{+}^{N} $, the two graphs share at least half of their spectrum, i.e.,
		\[\liminf_{T\to\infty}\frac{|\spec(\Gamma,\lv)\cap\spec(\Gamma',\lv)\cap[0,T]|}{|\spec(\Gamma,\lv)\cap[0,T]|}\ge\frac{1}{2}.\]
\end{enumerate}
\end{thm}
\begin{rem} We mention that the spectra of graphs with equal edge lengths share the same linear growth rate due to Weyl's law
	\[|\spec(\Gamma,\lv)\cap[0,T]|\asymp|\spec(\Gamma',\lv)\cap[0,T]|=\frac{\pi}{L}T+O(1),\qquad T\to\infty\]
	where $ L:=\sum_{j=1}^{N}\ell_{j} $; see \cite[p. 95]{BerKuc_graphs}. 
\end{rem}
The above two theorems descend from a dichotomy of scale-invariant vertex conditions, i.e., linear equations in $ \tr_{k}(f) $. Heuristically, up to some technicalities, a scale-invariant vertex condition either is always satisfied or it is (strongly and ergodically) generically never satisfied. This dichotomy can be generalized to any $ q(\tr_{k}(f))=0 $ condition, with $ q $ homogeneous polynomial. In fact, it can be further generalized to polynomial conditions on the trace space, $ q(\exp(ik\lv),\tr_{k}(f))=0 $ where $ q $ is a polynomial that is homogeneous in the $ \tr_{k}(f) $ variables. The homogeneity makes these conditions independent of the eigenfunction's normalization. The next two theorems rigorously state this dichotomy. 
 

\begin{thm}[Trace space genericity]\label{thm: polynomial vertex conditions} Let $ \Gamma $ be a graph satisfying Assumption \ref{ass: assumptions} that is not a mandarin. Let $ q $ be a polynomial on $ \C^{N}\times\C^{4N} $ that is homogeneous in the $ \C^{4N} $ coordinates. If there exist an $ \lv\in\R_{+}^{N} $ and an eigenpair $ (k^2,f) $ of $ (\Gamma,\lv) $ such that
	\begin{enumerate}
		\item[i)] $ k^2 $ is a non-zero simple eigenvalue.
		\item[ii)]  $ f $ is not supported on a loop, and $ q(\exp{(ik\lv)},\tr_{k}(f))\ne0 $. 
	\end{enumerate}
Then, the following properties of eigenpairs are both strongly and ergodically generic:
\begin{enumerate}
	\item $ k^2 $ is a simple eigenvalue, and
	\item $ q(\exp{(ik\lv)},\tr_{k}(f))\ne0 $ when $ f $ is not supported on a single loop.
\end{enumerate} 	
\end{thm}
Mandarin graphs are excluded in Theorem \ref{thm: polynomial vertex conditions} due to a certain reflection symmetry. A mandarin graph is symmetric to the reflection of all edges simultaneously, resulting in the swapping of the two vertices \cite{BanBerWey_jmp15}. A function is called \emph{symmetric} if it is invariant under this reflection, and \emph{anti-symmetric} if $ f\mapsto -f $. The orthonormal set of eigenfunctions can be chosen such that all eigenfunctions are either symmetric or anti-symmetric.
\begin{thm}[Trace space genericity for mandarins]\label{thm: polynomial vertex conditions mandarin} Let $ \Gamma $ be a mandarin graph of $ N\ge 3 $ edges. Let $ q $ be any polynomial on $ \C^{N}\times\C^{4N} $ that is homogeneous in the $ \C^{4N} $ coordinates. If there exist an $ \lv\in\R_{+}^{N} $ and an eigenpair $ (k^2,f) $ of $ (\Gamma,\lv) $ such that
	\begin{enumerate}
		\item[i)] $ k^2 $ is a non-zero simple eigenvalue.
		\item[ii)]  $ f $ is symmetric (resp. anti-symmetric), and $ q(\exp{(ik\lv)},\tr_{k}(f))\ne0 $. 
	\end{enumerate}
	Then, the following properties of eigenpairs are both strongly and ergodically generic:
	\begin{enumerate}
		\item $ k^2 $ is a simple eigenvalue, and
		\item $ q(\exp{(ik\lv)},\tr_{k}(f))\ne0 $ when $ f $ is symmetric (resp. anti-symmetric).
	\end{enumerate} 
\end{thm}
\begin{rem}
	Notice that Theorem \ref{thm: polynomial vertex conditions} distinguishes between symmetric and anti-symmetric eigenfunctions for mandarin graphs, and Theorem \ref{thm: polynomial vertex conditions} distinguishes between eigenfunctions that are supported on a loop and the rest of the eigenfunctions whenever a graph has loops. Both cases reflect a certain algebraic property of $ \Sigma(\Gamma) $; it is reducible for mandarin graphs and for graphs with loops. We elaborate on that in Section \ref{sec: irreduicible}. 
\end{rem}

	\section{The trace space and the secular manifold}\label{sec: Tspace and secular manifold}
	In this section we construct the trace space $ \TS(\Gamma) $ and the secular manifold $ \Sigma(\Gamma) $ as zero sets of polynomials, establishing their algebraic structure. We use the notation $ (\bz,\xv) $ for points in $ \T^{N}\times\C^{4N}\subset\C^{N}\times\C^{4N} $ so that $ (\bz,\xv)\in \TS(\Gamma) $ implies there exists $ \lv\in\R_{+}^{N} $ and some $ (k^2,f) $ eigenpair of $ (\Gamma,\lv) $, such that
	\[\bz=\exp(ik\lv),\quad\mbox{and}\quad\xv=\tr_{k}(f).\] 
	\begin{rem}
		We may assume that $ k\ne0 $. To see that, recall that $ k=0 $ is a simple eigenvalue for any (connected) $ (\Gamma,\lv) $ and its eigenfunction is constant $ f\equiv A $. The corresponding point $ (\bz,\xv) $ in such case has $ \bz=(1,1,\ldots,1) $ and $ \xv $ with all Dirichlet entries equal to $ A $ and zero Neumann entries. We get the same point $ (\bz,\xv) $ if we choose $ \lv=(2\pi,2\pi,\ldots,2\pi) $ with $ k=1 $ and eigenfunction $ f $ whose restrictions are 
		\[f|_{e_{j}}(t_{j})=A\cos(t_{j}),\qquad j=1,2,\ldots,N.\]
	\end{rem}
	The standard vertex conditions (see Definition \ref{def: standard vertex conditions}) can be written as a linear equation
	\[P_{\mathrm{std}}\tr_{k}(f)=0,\]
	where $ P_{\mathrm{std}} $ is a $ 2N\times4N $ matrix of rank $ 2N $. Recall the notation $ (A_{j}, B_{j}, C_{j}, D_{j}) $ for the restriction of $ \xv= \tr_{k}(f)$ to the edge $ e_{j} $, such that 
	\begin{equation}\label{eq: amplitudes}
		f|_{e_{j}}(t_{j})= A_{j}\cos(kt_{j})+B_{j}\sin(kt_{j})=C_{j}\cos(k(\ell_{j}-t_{j}))+B_{j}\sin(k(\ell_{j}-t_{j}))
	\end{equation}
\begin{lem}\label{lem: trace space polynomial equation}
	Let $ \Gamma $ be a graph with $ N $ edges. Then its trace space $ \TS(\Gamma) $ is equal to the set of $ (\bz,\xv)\in\T^{N}\times\C^{4N} $ which satisfies the following multi-linear\footnote{A multi-linear function is a multi-variable polynomial of degree one in each variable.} equations:
	\begin{enumerate}
		\item Vertex conditions: $\quad P_{\mathrm{std}}\xv =0 $.
		\item Edge conditions:  
			\begin{align}\label{eq: edge equation 1}
			A_{j}+iB_{j}- z_{j}(C_{j}-iD_{j}) & =0,\\
			C_{j}+iD_{j}- z_{j}(A_{j}-iB_{j}) & =0 \label{eq: edge equation 2},
		\end{align}
	for every edge $ e_{j} $.
	\end{enumerate}
In particular, $ \TS(\Gamma) $ is an algebraic subvariety of $ \T^{N}\times\C^{4N} $. 
Furthermore, the $ \xv $ fiber above a base point $ \bz\in\Sigma(\Gamma) $, i.e.,
\[ \TS_{\bz}(\Gamma):=\set{\xv\in\C^{4N}}{(\bz,\xv)\in\TS(\Gamma)}, \]
is a subspace of $ \C^{N} $ that has a basis of real vectors.
\end{lem}	
See Figure \ref{fig: parameters on graph} for example of the trace space coordinates assigned to a graph. 
 \begin{figure}
 	\includegraphics[width=0.75\paperwidth]{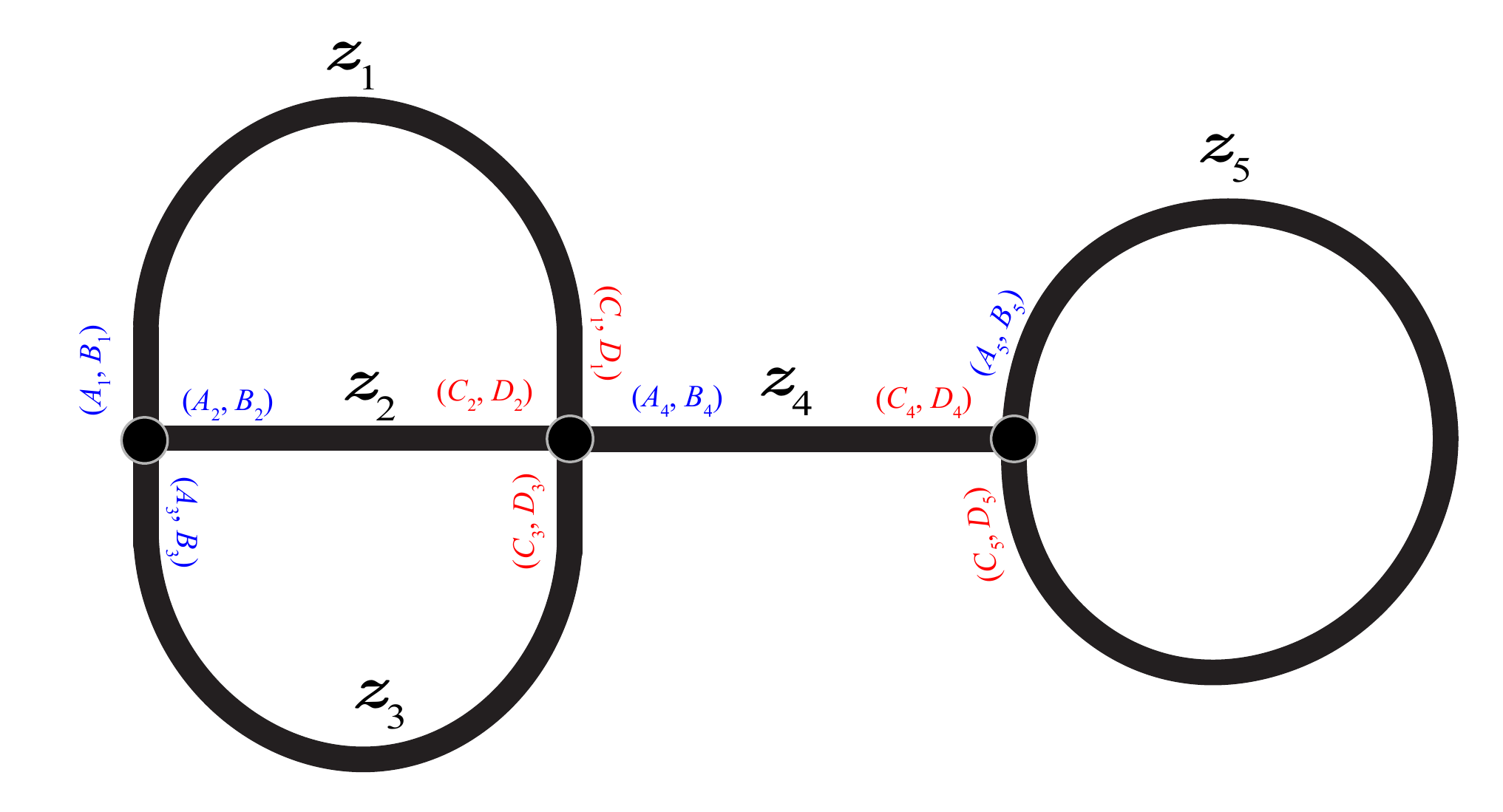}
 	
 	\caption{\label{fig: parameters on graph} An example of a possible assignment of the parameters $ z_{j} $ and $ (A_{j}, B_{j}, C_{j}, D_{j}) $ to the edges of a graph. }
 \end{figure}

\begin{proof}[Proof of Lemma \ref{lem: trace space polynomial equation}]
	Observe that any $ (\bz,\xv)\in\T^{N}\times\C^{4N} $ has some choice of $\lv\in\R_{+}^{N},~~k\in\R_{+} $ and $ f\in H^{2}(\Gamma,\lv) $ such that $ (\bz,\xv)=(\exp(ik\lv),\tr_{k}(f)) $. By definition $ (\bz,\xv)\in\TS(\Gamma) $ if and only if $ (k,f) $ is an eigenpair of $ (\Gamma,\lv) $. Hence, we want to show that a pair $ (k^2,f)\in\R_{+}\times H^2(\Gamma,\lv) $ is an eigenpair of $ (\Gamma,\lv) $ if and only if $ (\bz,\xv)=(\exp(ik\lv),\tr_{k}(f)) $ satisfy conditions (1) and (2) in Lemma \ref{lem: trace space polynomial equation}. Condition (1) is equivalent to $ f $ satisfying the standard vertex conditions on every vertex, so it is necessary. Assuming that $ f $ satisfies the vertex conditions, then $ (k^2,f) $ is an eigenpair of $ (\Gamma,\lv) $ if and only if
	\[-f''|_{e_{j}}=k^2 f|_{e_{j}},\qquad\mbox{for all}\quad j=1,2,\ldots,N.\]
	For every edge $ e_{j} $, a general solution to the ODE above can be written in two ways:
	\begin{enumerate}
		\item[i)] $\qquad f|_{e_{j}}(t_{j})= A_{j}\cos(kt_{j})+B_{j}\sin(kt_{j}), $
		\item[ii)] $ \qquad f|_{e_{j}}(t_{j})=C_{j}\cos(k(\ell_{j}-t_{j}))+B_{j}\sin(k(\ell_{j}-t_{j}))$.
	\end{enumerate}
			By comparing the two equivalent solutions we see that 
	\begin{equation}\label{eq: edge equation 0}
		\begin{pmatrix}
			A_{j} \\ B_{j}
		\end{pmatrix}=\begin{pmatrix}
			\cos(k\ell_{j}) & \sin(k\ell_{j}) \\ \sin(k\ell_{j}) & -\cos(k\ell_{j})
		\end{pmatrix}\begin{pmatrix}
			C_{j} \\ D_{j}
		\end{pmatrix}.
	\end{equation}
	If we left-multiply \eqref{eq: edge equation 0} by the invertible matrix
	\[\begin{pmatrix}
		1 & i\\
		z_{j} & -iz_{j}
	\end{pmatrix},\quad\mbox{with}\quad z_{j}=e^{ik\ell_{j}}, \] and rearrange, we get the two needed equations,
	\begin{align*}
		A_{j}+iB_{j}- z_{j}(C_{j}-iD_{j}) & =0,\\
		C_{j}+iD_{j}- z_{j}(A_{j}-iB_{j}) & =0 .
	\end{align*}
	This proves that a point $ (\bz,\xv) $ lies in $ \TS(\Gamma) $ if and only if it satisfies conditions (1) and (2).\\
	
	In the above proof we have seen that condition (2) is equivalent to satisfying \eqref{eq: edge equation 0} for all edges. Fix $ \bz\in\Sigma(\Gamma) $, so that $ \TS_{\bz}(\Gamma) $ is the set of vectors $ \xv\in\C^{4N} $ that satisfy $ P_{\mathrm{std}}\xv =0 $ and \eqref{eq: edge equation 0} for all edges. This are linear equations with real coefficients and therefore $ \TS_{\bz}(\Gamma) $ is a subspace of $ \C^{4N} $ that has a basis of real vectors. 
\end{proof}
\begin{rem}
	The fact that $ \TS_{\bz}(\Gamma) $ has a basis of real vectors reflects the fact that the Laplacian and the standard vertex conditions are real and therefore any eigenspace has a basis of real eigenfunctions. If we consider the magnetic Laplacian, then the eigenspaces are no longer spanned by real eigenfunctions and therefore so does their traces. A possible future generalization of the trace space to a magnetic Laplacian would probably require to replace $ z_{j}\mapsto e^{i\alpha_{j}}z_{j} $ in \eqref{eq: edge equation 1} and $ z_{j}\mapsto e^{-i\alpha_{j}}z_{j} $ in \eqref{eq: edge equation 2}, in which case $ \TS_{\bz}(\Gamma) $ will no longer have a basis of real vectors.    
\end{rem}

\begin{rem}\label{rem: edge equation solutions}
	A solution to \eqref{eq: edge equation 1} and \eqref{eq: edge equation 2}, given a fixed $ z_{j} $ with $ |z_{j}|=1 $, satisfies
	\[\|(A_{j},B_{j})\|=\|(C_{j},D_{j})\|.\]
	If $ A_{j}=B_{j}=C_{j}=D_{j}=0 $, the solution is independent of $ z_{j} $.   
\end{rem}
At this point, we can see that the secular manifold 
\[\Sigma(\Gamma):=\set{\bz\in\T^{N}}{\exists \xv\in \C^{4N}\quad\mbox{s.t.}\quad (\bz,\xv)\in\TS(\Gamma)},\]
is the projection of an algebraic variety (i.e. common zero set of polynomials) but it is not clear that $ \Sigma(\Gamma) $ is itself a zero set of a polynomial. However, the standard construction of $ \Sigma(\Gamma) $ is as a zero set of a polynomial $ P_{\Gamma} $ restricted to $ \T^{N} $, as was first done in \cite{BarGas_jsp00}. We will refer to this construction as the ``unitary approach", and we will now show how the trace space is described in that language. The unitary approach, was motivated by scattering systems of propagating waves \cite{KotSmi_prl97}. Given an eigenpair $ (k^2,f) $, the restriction of $ f $ to the edge $ e_{j} $ can be written in another equivalent form 
\begin{equation}\label{eq: f in terms of complex amplitudes}
	f|_{e_{j}}(t_{j})=a_{j}e^{-ik(\ell_{j}-t_{j})}+b_{j}e^{-ikt_{j}},
\end{equation}
which is interpreted as incoming and outgoing waves with complex amplitudes $ a_{j} $ and $ b_{j} $. Denote the vector of amplitudes $ f $ by $ \textbf{a}:=(a_{1},\ldots,a_{N},b_{1},\ldots,b_{N})\in \C^{2N} $. There exists a real orthogonal $ 2N\times2N $ matrix $ S $ such that
\begin{equation}\label{eq: eigenvalue equation with U}
	\diag(\bz,\bz)S\textbf{a}=\textbf{a}\iff f\in\eig(\Gamma,\lv,k),\qquad \bz=\exp(ik\lv).
\end{equation} 
See \cite{GnuSmi_ap06} or Section 2.1 of \cite{BerKuc_graphs} for a proof. In the above, $ \diag(\bz,\bz) $ is a diagonal $ 2N\times 2N $ matrix with $ (\bz,\bz)=(z_{1},\ldots,z_{N},z_{1},\ldots,z_{N})$ along the diagonal.
\begin{rem}
	The fixed matrix $ S $, is often called the \emph{bond scattering matrix} of $ \Gamma $ and is given explicitly in \cite{BerKuc_graphs,GnuSmi_ap06}. The fact that it is real orthogonal and $ k $ independent is due to the standard vertex conditions. For general vertex conditions, $ S $ is unitary and $ k $ dependent.
\end{rem}  
\begin{defn}
	Given a graph $ \Gamma $ and its associated $ S $ matrix, let 
	\[U(\bz):=\diag(\bz,\bz)S.\]
	The \emph{characteristic polynomial} or \emph{secular polynomial} of the graph $ \Gamma $ is defined by 
	\[P_{\Gamma}(\bz):=\det(\mathbb{I}_{2N}-U(\bz)),\]
	where $ \mathbb{I}_{2N} $ is the $ 2N\times2N $ identity matrix.
\end{defn}
\begin{lem}\cite{BarGas_jsp00,BerWin_tams10,CdV_ahp15}\label{lem: sec man in U}
	The secular manifold is the zero set of $ P_{\Gamma}(\bz) $ restricted to $ \T^{N} $, i.e., 
	\[\Sigma(\Gamma)=\set{\bz\in\T^{N}}{P_{\Gamma}(\bz)=0}.\]
	It is a subvariety of $ \T^{N} $ of real dimension $ N-1 $. We partition $ \Sigma(\Gamma) $ into two subsets, 
	\begin{align*}
		\msing(\Gamma):= & \set{\bz\in\T^{N}}{P_{\Gamma}(\bz)=0,~~\nabla P_{\Gamma}(\bz)=0},\\
		\mreg(\Gamma):= & \set{\bz\in\T^{N}}{P_{\Gamma}(\bz)=0,~~\nabla P_{\Gamma}(\bz)
			\ne0}.
	\end{align*}
		This partition captures the multiplicity of the eigenvalues. An eigenvalue $ k^2\ne0 $ of $ (\Gamma,\lv) $ is simple when $ \exp(ik\lv)\in\mreg(\Gamma) $, and has multiplicity when $ \exp(ik\lv)\in\msing(\Gamma) $. 
\end{lem}
	The statements of this lemma can be found for example in Theorem 1.1 of \cite{CdV_ahp15}, up to working with the flat torus $ \R^{N}/2\pi\Z^{N} $ instead of $ \T^{N} $. Nevertheless we will provide a proof later using the trace space, for completeness. 
\begin{rem}
	As a subvariety $ \Sigma(\Gamma) $ can have singular points. The set $ \msing(\Gamma) $ is the set of singular points of $ \Sigma(\Gamma) $. This is not immediate from $ \nabla P_{\Gamma}(\bz)=0 $ but requires the fact that $ P_{\Gamma} $ has no square factors, which we will show in Theorem \ref{Thm: Sarnak-Kurasov}. The complement $ \mreg(\Gamma) $ is the set of regular points of $ \Sigma(\Gamma) $ and is therefore a real analytic manifold. 
\end{rem}
Lemma \ref{lem: sec man in U} characterize $ \Sigma(\Gamma) $ as the set of $ \bz\in\T^{N} $ for which $ U(\bz) $ has eigenvalue equal to $ 1 $. We will show that the trace space is ``essentially" the pairs $ (\bz,\textbf{a}) $ such that $ U(\bz)\textbf{a}=\textbf{a}$. To this end, we need to relate $ \textbf{a} $ to a trace $ \xv $. We define the following $ 2N\times2N $ and $ 4N\times 2N $ matrices
\[J:=\begin{pmatrix}
	0 & \mathbb{I}_{N}\\ \mathbb{I}_{N} & 0
\end{pmatrix},\qquad M:=\begin{pmatrix}
	S+J\\ i(S-J)
\end{pmatrix},\]
where $ \mathbb{I}_{N} $ is the $ N $ dimensional identity matrix.
\begin{lem}\label{lem: trace space in U}
	Given a graph $ \Gamma $ of $ N $ edges, its trace space is equal to 	\[\TS(\Gamma)=\set{(\bz,M\textbf{a})}{\bz\in\Sigma(\Gamma)~~\mbox{ and  }~~U(\bz)\textbf{a}=\textbf{a}}.\]
\end{lem}
Equations \eqref{eq: eigenvalue equation with U} and \eqref{eq: f in terms of complex amplitudes} provides an isomorphism between an eigenspace $ \eig(\Gamma,\lv,k) $ and the kernel $ \ker(1-U(\bz)) $ for $ \bz=\exp(ik\lv) $, and as a corollary of the above lemma, both are isomorphic to the corresponding fiber. 
\begin{cor}\label{cor: isomorhpism}
		Let $ k^2>0 $ be an eigenvalue of $ (\Gamma,\lv) $ and let $ \bz=\exp(ik\lv) $. Then the $ \xv $ fiber above $\bz  $ be written in two ways:
		\begin{enumerate}
			\item[i)] $ \TS_{\bz}(\Gamma)=M(\ker(\mathbb{I}_{2N}-U(\bz))) $, and
			\item[ii)] $ \TS_{\bz}(\Gamma)=\tr_{k}(\eig(\Gamma,\lv,k)) $.  
		\end{enumerate} 
	Moreover, these are isomorphisms.
\end{cor}
\begin{proof}[Proof of Corollary \ref{cor: isomorhpism}]
	It is enough to show that $ \tr_{k}:\eig(\Gamma,\lv,k)\to\C^{4N} $ and $ M:\C^{2N}\to\C^{4N} $ are injective. The injectivity of $ \tr_{k} $ follows from \eqref{eq: amplitudes}. For $ M $, let $ \textbf{a}\in\C^{2N} $ such that $ M\textbf{a}=0 $. Since $ J^2=\mathbb{I}_{2N} $, the equation $ JM\textbf{a}=0 $ can be written as 
	\begin{align*}
		JS\textbf{a}+\textbf{a} & =0,\quad\mbox{and}\\
		i(JS\textbf{a}-\textbf{a}) & =0,
	\end{align*}
	namely $ \textbf{a}=JS\textbf{a}=-JS\textbf{a} $ so $ \textbf{a}=0 $.  
\end{proof}
\begin{proof}[Proof of Lemma \ref{lem: trace space in U}]
	Let $ \bz\in\Sigma(\Gamma) $ and consider a choice of an $ \lv\in\R_{+}^{N} $ and a non zero eigenvalue $ k^2 $ of $ (\Gamma,\lv) $ such that $ \bz=\exp(ik\lv) $. We need to show that
	\[\TS_{\bz}(\Gamma)=M(\ker(\mathbb{I}_{2N}-U(\bz))).\]
	By definition, $  \TS_{\bz}(\Gamma)=\tr_{k}(\eig(\Gamma,\lv,k))  $ so it is enough to show that for any eigenfunction $ f\in\eig(\Gamma,\lv,k)  $ with amplitudes vector $ \textbf{a} $ given by \eqref{eq: f in terms of complex amplitudes}, the following relation holds 
	\[\tr_{k}(f)=M\textbf{a}.\]
	By comparing \eqref{eq: f in terms of complex amplitudes} and \eqref{eq: amplitudes} which describe the same restriction $ f|_{e_{j}} $ in terms of $ (A_{j},B_{j},C_{j},D_{j}) $ and $ (a_{j},b_{j}) $, we get the following relation. 
\begin{equation}\label{eq:amplitudes to trace 1}
	\begin{pmatrix}
		A_{j} \\ B_{j}\\ C_{j} \\ D_{j}
	\end{pmatrix}  =	\begin{pmatrix}
		a_{j}z_{j}^{-1}+b_{j} \\ ia_{j}z_{j}^{-1}-ib_{j}\\
		a_{j}+b_{j}z_{j}^{-1} \\ -ia_{j}+ib_{j}z_{j}^{-1}
	\end{pmatrix}=\begin{pmatrix}
		z_{j}^{-1} & 1 \\ iz_{j}^{-1} & -i\\
		1 & z_{j}^{-1} \\ -i & iz_{j}^{-1}
	\end{pmatrix}\begin{pmatrix}
		a_{j} \\ b_{j}
	\end{pmatrix}. 
\end{equation}
Since $ \textbf{a}$ satisfies $ \diag(\bz,\bz)S\textbf{a}=\textbf{a} $, then for any $ j=1,2,\ldots,N $
\begin{align*}
	(S\textbf{a})_{j}=z_{j}^{-1}a_{j},&\quad (S\textbf{a})_{j+N}=z_{j}^{-1}b_{j},\\
	(J\textbf{a})_{j}=b_{j},&\quad\mbox{and}\quad (J\textbf{a})_{j+N}=a_{j},
\end{align*}
and so \eqref{eq:amplitudes to trace 1} can be written, for all edges simultaneously, as   
\[\xv=\tr_{k}(f)=M\textbf{a}.\]
\end{proof}
We may now prove Lemma \ref{lem: sec man in U} using Corollary \ref{cor: isomorhpism}.
\begin{proof}[Proof of Lemma \ref{lem: sec man in U}]
	The fact that $ \Sigma(\Gamma) $ is the zero set of $ P_{\Gamma}(\bz)=\det(\mathbb{I}_{2N}-U(\bz)) $ follows from \eqref{eq: eigenvalue equation with U}.\\
	
	Let $ k^2>0 $ be an eigenvalue of $ (\Gamma,\lv) $ and let $ \bz=\exp()ik\lv $. Let $ d= \dim(\ker(\mathbb{I}_{2N}-U(\bz)))  $. By Corollary \ref{cor: isomorhpism}, we know that $ d\ge 1 $, and that  $ k^2 $ is simple if $ d=1 $ and is multiple if $ d>1$. We need to show that $ d>1 $ if and only if both $ P_{\G} $ and $ \nabla P_{\Gamma} $ vanish at $ \bz $. To compute the derivatives of $ P_{\Gamma}(\bz):=\det(1-U(\bz)) $, we use the Jacobi formula for the derivative of a determinant $ \det(A) $ in terms of the adjugate matrix $ \adj (A) $, 
	\[\nabla P_{\Gamma}(\bz)=\mathrm{Trace}[\adj (1-U(\bz))\nabla(1-U(\bz))].\]
	The adjugate matrix $ \adj (A) $ is a matrix whose entries are minors of $ A $ and it satisfies $ A\adj (A)=\adj (A)A=\det(A)I $ where $ I $ is the identity matrix. In particular, it satisfies
	\begin{enumerate}
		\item If $ \dim(\ker(A))=0 $ then $ \adj (A)=\det(A)A^{-1} $.
		\item If $ \dim(\ker(A))>1 $, then $ \adj (A)=0 $.
		\item If $ \dim(\ker(A))=1 $ then $ \adj (A) $ is a rank one matrix proportional to the orthogonal projection on $\ker(A) $. 
	\end{enumerate}
	By Substituting $ A=\mathbb{I}_{2N}-U(\bz) $, Property (2) provides one side of the if and only if, namely that $ \nabla P_{\Gamma}(\bz)=0 $ when $ d>1$. For the other side, we want to show that $ d=1 $ implies $ \nabla P_{\Gamma}(\bz)\ne0 $. Assume that $ d=\dim(\ker(\mathbb{I}_{2N}-U(\bz)))=1 $  and let $ \textbf{a} $ be the (unique up to a phase) normalized vector in $\ker(\mathbb{I}_{2N}-U(\bz)) $. By property (3), 
	\begin{equation}\label{eq: adjugate}
		\adj (\mathbb{I}_{2N}-U(\bz)=c_{\bz}\textbf{a}\textbf{a}^{*},\qquad c_{\bz}\in\C\setminus\{0\}.
	\end{equation}
	Since $ \textbf{a}\ne 0 $, then $ |a_{j}|^2+|b_{j}|^2\ne0 $ for some $ j=1,2,\ldots N $. We calculate, 
	\begin{align*}
		\frac{\partial}{\partial_{z_{j}}}P_{\Gamma}(\bz)=&\mathrm{Trace}(\adj (1-U(\bz))(\frac{\partial}{\partial_{z_{j}}}\diag(\bz,\bz))S)\\
		=& c_{z} \textbf{a}^{*}(\frac{\partial}{\partial_{z_{j}}}\diag(\bz,\bz))S\textbf{a}\\
		= & c_{z}(\bar{a}_{j}(S\textbf{a})_{j}+\bar{b}_{j}(S\textbf{a})_{j+N})\\
		= & \frac{c_{z}}{z_{j}}(|a_{j}|^2+|b_{j}|^2)\ne 0,
	\end{align*}
	where in the last equality we used that $ (S\textbf{a})_{j}=\frac{a_{j}}{z_{j}} $ and $ (S\textbf{a})_{j+N}=\frac{b_{j}}{z_{j}} $, as we assumed that $ \diag(\bz,\bz)S\textbf{a}=\textbf{a} $. This argument proves that $ \nabla P_{\Gamma}\ne 0 $ when $ \dim(\ker(\mathbb{I}_{2N}-U(\bz)))=1 $.\\ 
\end{proof}
In the proof of Lemma \ref{lem: sec man in U} we have shown that
\[\frac{\partial}{\partial_{z_{j}}}P_{\Gamma}(\bz)=\frac{c_{z}}{z_{j}}(|a_{j}|^2+|b_{j}|^2),\qquad \frac{c_{z}}{z_{j}}\ne 0.\]
This leads to the next lemma.  
 \begin{lem}\label{lem: vanishihng on an edge}
	Let $ f $ be an eigenfunction of $ (\Gamma,\lv) $ with a simple non-zero eigenvalue $ k^2 $, so that $ \bz=\exp(ik\lv)\in\mreg(\Gamma) $. Then, for any edge $ e_{j} $,
	\[f|_{e_{j}}\equiv 0\iff  \frac{\partial}{\partial_{z_{j}}}P_{\Gamma}(\bz)=0.\]  
\end{lem}
The statement of Lemma \ref{lem: vanishihng on an edge} can be attributed\footnote{Both \cite{BanLev_prep16} and \cite{CdV_ahp15} showed that if $ (k^2,f) $ is an eigenpair of $ (\Gamma,\lv) $ and $ k^2 $ is simple, then $ \frac{\partial}{\partial \ell_{j}}k^2=0 $ if and only if $ f|_{e_{j}}\equiv 0 $. This statement can be shown to be equivalent to Lemma \ref{lem: vanishihng on an edge}} to \cite{BanLev_prep16,CdV_ahp15}. This lemma describes a property of the trace vector in terms of the $ \bz $ coordinates. The next lemma provides a stronger statement, constructing a rank-one matrix $ A(\bz) $, for $ \bz\in \mreg(\bz) $, that is proportional to the orthogonal projection onto the $ \xv $ fiber $ \TS(\Gamma)_{\bz}  $. 
\begin{lem}\label{lem: trace as Az}
	Consider the matrix $ M $ as in Lemma \ref{lem: trace space in U}. Define the $ \bz $ dependent matrix \[A(\bz):=M\adj (1-U(\bz))M^{*}.\]
	The  $4N\times 4N $ matrix $ A(\bz) $ has the following properties:
	\begin{enumerate}
		\item Its entries are polynomials in $ \bz $.
		\item If $ \bz\in\msing(\Gamma) $, then $ A(\bz)=0 $.
		\item If  $ \bz\in\mreg(\Gamma) $, then $ A(\bz)$ is proportional to the rank-one matrix $ \xv\xv^{*} $, for any $ \xv\ne0 $ such that $ (\bz,\xv)\in \TS(\Gamma) $. That is,
		\[A(\bz)=c_{\bz,\xv}\xv\xv^{*},\qquad c_{\bz,\xv}\in\C\setminus\{0\}.\]  
	\end{enumerate} 
\end{lem}
\begin{proof}
	Since $ M $ is constant, the entries of $ A(\bz) $ are linear in the entries of $ \adj (\mathbb{I}_{2N}-U(\bz)) $ which are minors of $ (\mathbb{I}_{2N}-U(\bz))=(\mathbb{I}_{2N}-\diag(\bz,bz)S) $ and hence polynomials in $ \bz $. If $ \bz\in\msing(\Gamma) $ then $ \dim(\ker(\mathbb{I}_{2N}-U(\bz)))>1 $ and hence $ \adj (\mathbb{I}_{2N}-U(\bz))=0 $, so $ A(\bz)=0 $. Now assume that $ \bz\in\mreg(\Gamma) $ so $ \dim(\ker(\mathbb{I}_{2N}-U(\bz)))=1 $, and let $ \xv\ne0 $ such that $ (\bz,\xv)\in \TS(\Gamma) $. According to \eqref{eq: adjugate}, given a normalized vector $ \textbf{a}\in\ker(\mathbb{I}_{2N}-U(\bz)) $, 
	\[	\adj (\mathbb{I}_{2N}-U(\bz)=\tilde{c}_{\bz}\textbf{a}\textbf{a}^{*},\]
	for some non-zero scalar $ \tilde{c}_{\bz} $. Since $ \ker(\mathbb{I}_{2N}-U(\bz)) $ is spanned by $ \textbf{a} $, then by Lemma \ref{lem: trace space in U},
	\[\xv=Mc'\textbf{a},\]
	for some non zero $ c' $. Denote the non-zero constant $ c_{\bz,\xv}=(c')^{2}\tilde{c}_{\bz} $ so that,
	\[\xv\xv^{*}=(c')^2 M\textbf{a}\textbf{a}^{*}M^{*}=c_{\bz,\xv}A(\bz).\]
\end{proof}
	\section{Graph reflection symmetries and the irreducible structure of the secular manifold}\label{sec: irreduicible}
	Consider the notation $ Z(p) $ for the zero set in $ \C^{N} $ of a given polynomial $ p $. The secular manifold, according to Lemma \ref{lem: sec man in U}, can be written as 
	\[\Sigma(\Gamma)=Z(P_{\Gamma})\cap\T^{N},\]
	in terms of the characteristic polynomial $ P_{\Gamma} $. We will show in this section that whenever $ P_{\Gamma} $ is irreducible (in the ring of polynomials $ \C[z_{1},z_{2}\ldots,z_{N}] $), any subvariety $ Z(q)\cap\Sigma(\Gamma) $ either has positive codimension in $ \Sigma(\Gamma) $ or it is equal to $ \Sigma(\Gamma) $. We say in such case that $ \Sigma(\Gamma) $ is irreducible. It was conjectured by Colin de Verdière that $ \Sigma(\Gamma) $ is irreducible if and only if $ (\Gamma,\lv) $ admits no $ \lv $-independent isometries (see the question prior to Proposition 1.1. in \cite{CdV_ahp15}). The main purpose of this section is to present a recent result of Kurasov and Sarnak \cite{KurasovBook,KurSar2022} which characterize the irreducible structure of $ P_{\Gamma} $ and proves the irreducibility conjecture mentioned above. This result is presented in Theorem \ref{Thm: Sarnak-Kurasov}. To state it, let us first characterize the $ \lv $-independent isometries in terms of \emph{reflection symmetries}. Given a map $ R:(\Gamma,\lv)\to (\Gamma,\lv) $ that sends every edge to itself, its restriction to every edge $ e_{j} $ is a map $ R|_{e_{j}}:[0,\ell_{j}]\to[0,\ell_{j}] $.
	\begin{defn}\label{def: reflection}
		We say that $ R:(\Gamma,\lv)\to (\Gamma,\lv) $ is a \emph{reflection symmetry} of $ \Gamma $ if 
		\begin{enumerate}
			\item $ R $ sends every edge $ e_{j} $ to itself, either by the identity map or by a reflection
			\[R|_{e_{j}}(x_{j})=\ell_{j}-x_{j}.\] 
			\item $ R $ preserves the graph structure, i.e., if an edge $ e $ is adjacent to a vertex $ v $, then $ R(e) $ is adjacent to $ R(v) $.
		\end{enumerate}
	\end{defn} 
It is not hard to show, as was already mentioned \cite{CdV_ahp15}, that an $ \lv $-independent isometry is a reflection symmetry, and can only happen if a graph has loops or is a mandarin graph.
\begin{lem}\label{lem: Reflections and isometries}\cite{CdV_ahp15,KurSar2022}
	An $ \lv $-independent isometry of $ (\Gamma,\lv) $ is a reflection symmetry, and vice versa. Furthermore, there are only two types of graphs that satisfy Assumption \ref{ass: assumptions} and have non-trivial reflection symmetries: 
		\begin{enumerate}
		\item \textbf{Mandarin graphs:} If $ \Gamma $ is a mandarin, then it has exactly one non-trivial reflection symmetry $ R $. It is a reflection on every edge. 
		\item \textbf{Graphs with loops:} If $ \Gamma $ has a loops, then for any loop $ e_{j} $ there is a reflection symmetry $ R_{j} $ acting by reflection on $ e_{j} $ and identity on all other edges. The group of reflection symmetries of $ \Gamma $ is generated by the loop reflections $ R_{j} $ for $ e_{j}\in\EL $.
	\end{enumerate}
\end{lem}
We present a short proof for completeness.
	\begin{proof}
		Let $ R:(\Gamma,\lv)\to(\Gamma,\lv) $ be an $ \lv $-independent isometry. As an isometry between one-dimensional Riemannian manifolds with singularities, $ \R $ sends singular points to singular points and line segments to line segments of the same lengths. Hence, $ R $ sends vertices to vertices and each edge is mapped to and edge of the same length. Since $ R $ is $ \lv $ independent then it must send each edge to itself. In particular, for any edge $ e_{j} $, the restriction $ R|_{e_{j}}:[0,\ell_{j}]\to[0,\ell_{j}] $ is an isometry, and is therefore either the identity or a reflection. To conclude that $ R $ is a reflection symmetry, let $ e $ be an edge adjacent to a vertex $ v $, and take a sequence of points $ x_{n}\in e $ converging to $ v $. Then, the sequence $ R(x_{n})  $ lies in $ R(e) $ and converges to $ R(v) $ by the isometry, and hence $ R(e) $ is adjacent to $ R(v) $. We conclude that $ R $ is a reflection symmetry.\\
		
		Now, let $ \Gamma $ satisfy Assumption \ref{ass: assumptions}. Assume that $ R $ is a reflection symmetry that acts by reflection on an edge $ e $, and let us deduce the action on the rest of the edges. We consider two cases.
		\begin{enumerate}
			\item If $ e $ is not a loop, let $ v,u $ be its distinct vertices and notice that $ R(v)=u $ and vice-versa. We may conclude that every edge adjacent to one of the vertices $ u,v $ is adjacent to both of them. For example, if $ e' $ is adjacent to $ u $ then $ R(e')=e' $ is adjacent to $R(u)=v $. It follows that $ u $ and $ v $ are connected one to the other and cannot be connected to any other vertex. As we assume the graph is connected, $ \Gamma $ has only two vertices, and the previous argument shows that every edge of $ \Gamma $ connects $ u $ to $ v $, so $ \Gamma $ is a mandarin graph, and $ R $ acts by reflection on all edges.
			\item  If $ e $ is a loop, then $ R|_{e_{j}} $ is the identity for every $ e_{j} $ which is not a loop, by (1). On any other loop $ e' $, $ R|_{e'} $ can be either a reflection or identity, independently of its action on $ e $.
		\end{enumerate} 
	\end{proof} 
The irreducibility theorem of Kurasov and Sarnak can now be stated. 
\begin{thm}\label{Thm: Sarnak-Kurasov}\cite{KurasovBook,KurSar2022}
	Let $\Gamma$ be a graph that satisfies Assumption \ref{ass: assumptions}, then $\PG\in\C[z_{1},\ldots,z_{N}]$ is irreducible if and only if $ \Gamma $ has no reflection symmetries. Moreover, if $ \Gamma $ has a reflection symmetry, then $\PG$ factors as follows:
	\begin{enumerate}
		\item If $ \Gamma $ has loops, then
		\begin{equation}\label{eq: PG for loops}
			\PG(\bz)=P_{\Gammasym}(\bz)\prod_{e_{j}\in\EL}(1-z_{j}),
		\end{equation}
	where $ \EL $ is the set of loops, and $ P_{\Gammasym}(\bz) $ is irreducible. 
		\item If $ \Gamma $ is a mandarin graph, then
		\begin{equation}\label{eq: PG for mandarins}
			\PG(\bz)=P_{M,\mathrm{s}}(\bz)P_{M,\mathrm{as}}(\bz),
		\end{equation}
		where both $P_{M,\mathrm{s}}$ and $ P_{M,\mathrm{as}} $ are irreducible multi-linear polynomials.
		\[P_{M,\mathrm{s}}(\bz):=\sum_{j=1}^{E}(z_{j}-1)\prod_{i\ne j}(z_{i}+1),\qquad P_{M,\mathrm{as}}(\bz):=\sum_{j=1}^{E}(z_{j}+1)\prod_{i\ne j}(z_{i}-1). \]
	\end{enumerate}
\end{thm} 
\begin{rem}\label{rem: degree of PG}
	For later use we mention that $ \PG $ has degree $ 2 $ in every $ z_{j} $ (as shown in \cite{KurasovBook,KurSar2022} for example). Therefore, if $ \Gamma $ has loops, then the degree  of $ P_{\Gammasym}(\bz) $ in $ z_{j} $ is one when $ e_{j} $ is a loop and two otherwise.  
\end{rem}
The notations $ P_{\Gammasym}, P_{M,\mathrm{s}}$ and $ P_{M,\mathrm{as}} $ do not appear in \cite{KurasovBook,KurSar2022}. We introduce these notation to emphasize that the symmetry type of eigenfunctions in $ \eig(\Gamma,\lv,k) $ is dictated by the $ \PG $ factors that vanish at $ \bz=\exp(ik\lv) $. We elaborate on that in Subsection \ref{sec: subsection Reflections}, but first, let us discuss some applications of Theorem \ref{Thm: Sarnak-Kurasov}. 
\subsection{Applications of the irreducibility}
Recall the notation $ Z(p):=\set{\bz\in\C^{N}}{p(\bz)=0} $. The next lemma shows that if $ p $ is either $ P_{\Gamma} $ or an irreducible factor of $ P_{\Gamma} $, then $ Z(p)\cap\T^{N} $ is Zariski dense in $ Z(p) $.    
\begin{lem}\label{lem: zariski}
	Let $ \Gamma $ be a graph satisfying Assumption \ref{ass: assumptions} and define $ p\in\C[z_{1},\ldots,z_{N}] $ as follows. If $ P_{\Gamma} $ is irreducible, let $ p= P_{\Gamma} $, and if $ P_{\Gamma} $ is reducible, let $ p $ be an irreducible factor of $ P_{\Gamma} $. Then, given any polynomial $ q\in\C[z_{1},\ldots,z_{N}]  $, exactly one of the following holds.
	\begin{enumerate}
		\item Either $ Z(q)\cap Z(p)\cap\T^{N} $ has real dimension at most $ N-2 $, i.e., positive codimension in $ \Sigma(\Gamma) $, or
		\item $ p $ is a factor of $ q $, in which case $ q(\bz)=0 $ for every $\bz\in Z(p)\cap\T^{N}\subset \Sigma(\Gamma) $. 
	\end{enumerate}  
\end{lem}
\begin{proof}
	The terminology ``toral polynomial" was introduced in \cite{Agl2006toral}(see Definition 2.2 and Proposition 2.1) to describe a polynomial $ p $ with the property that any polynomial $ q $ that vanish on $ Z(p)\cap\T^{N} $ must vanish on $ Z(p) $. Denote the polydisc
	\[D^{N}:=\set{\bz\in\C^{N}}{|z_{j}|<1,~~j=1,2,\ldots,N},\] and its inverse
	\[(\C\setminus\overline{D})^{N}:=\set{\bz\in\C^{N}}{|z_{j}|>1,~~j=1,2,\ldots,N} .\]
	Theorem 3.5 in \cite{Agl2006toral} states that if $ Z(p) $ is disjoint from $ D^{N}\cup (\C\setminus\overline{D})^{N}$ then $ p $ is toral. We claim that this is the case for $ P_{\Gamma} $, as Sarnak and Kurasov discuss in \cite{kurasov2020stable}. To prove it, recall that $ P_{\Gamma}(\bz):=\det(1-\diag(\bz,\bz)S) $ and that $ S $ is real orthogonal. Assume $ \diag(\bz,\bz)S\textbf{a}=\lambda\textbf{a} $ for some non-zero $ \textbf{a} $, and notice that $ \|S\textbf{a}\|=\|\textbf{a}\| $ since $ S $ is orthogonal. If $ \bz\in D^{n} $, then $ \diag(\bz,\bz) $ is strictly contracting, i.e. $ \|\diag(\bz,\bz)\textbf{a} \|<\|\textbf{a}\| $. Hence $ \|\lambda\textbf{a}\|<\|\textbf{a}\| $ so $ \lambda\ne 1 $. It follows that $ P_{\Gamma}(\bz)\ne0 $ in that case. For the other case, $ \bz\in (\C\setminus\overline{D})^{N} $, a similar proof, using $ \|\diag(\bz,\bz)\textbf{a} \|>\|\textbf{a}\| $ in this case, would give that
	$ P_{\Gamma}(\bz)\ne0 $. We may conclude that $ Z(P_{\Gamma}) $ is disjoint from $ D^{N}\cup (\C\setminus\overline{D})^{N} $, and therefore $ P_{\Gamma} $ is toral. Let $ p $ be an irreducible polynomial which is either $ P_{\Gamma} $ or one of its factors. Clearly, $ Z(p)\subset Z( P_{\Gamma}) $ is also disjoint from $ D^{N}\cup (\C\setminus\overline{D})^{N} $ and therefore $ p $ is toral.\\
	
	Let $ q $ be any polynomial. If $ q $ vanish entirely on $ Z(p)\cap\T^{N} $, then we may deduce from $ p $ being toral that $ q $ vanish on $ Z(p) $, and since $ p $ is irreducible, it follows that $ p $ is a factor of $ q $. \\
	
	Now assume that $ p $ is not a factor of $ q $, so the common zero set $ V=Z(p)\cap Z(q) $ is a variety of complex dimension $ N-2 $. By Lemma \ref{lem: real dimension bounded by complex dimension}, we may deduce that $ Z(p)\cap Z(q)\cap\T^{N}=V\cap\T^{N} $ has real dimension at most $ N-2 $.    
\end{proof}
According to \ref{Thm: Sarnak-Kurasov}, either $ P_{\Gamma} $ is irreducible, or it is reducible but each factors appears once. We say that $ P_{\Gamma} $ has no square factors. It is a simple observation (by counting degrees) that an irreducible polynomial does not share factors with any of its derivatives. It follows that a reducible polynomial that shares a common factor with all of its derivatives must have a square factor. Since $ P_{\Gamma} $, even if reducible, has no square factors, then it must have at least one derivative $ \frac{\partial}{\partial z_{j}} P_{\Gamma} $ without a common factor. Then next corollary follows by applying Lemma \ref{lem: zariski} to the singular set
\[\msing(\Gamma)\subset Z(P_{\Gamma})\cap Z(\frac{\partial}{\partial z_{j}} P_{\Gamma})\cap\T^{N}.\]
\begin{cor}\label{cor: singular set}
	For any $ \Gamma $ satisfying Assumption \ref{ass: assumptions}, the singular set $\msing(\Gamma)$ is a subvariety of $ \Sigma(\Gamma) $ of positive codimension. That is, $\msing(\Gamma)$ has real dimension at most $ N-2 $. 
\end{cor}
\begin{rem}
	Corollary \ref{cor: singular set} is well known and can be found for example in \cite{CdV_ahp15}. However, to the best of our knowledge this is the first proof that does not rely on Friedlander's simplicity result \cite{Fri_ijm05}, which means that the results of this paper do not rely on \cite{Fri_ijm05}.
\end{rem}
The next lemma that we can now prove will be the main ingredient in proving Theorem \ref{thm: disjoint spectrum} regarding the common spectrum of metric graphs. 
\begin{lem}[No common factors]\label{lem: no common factors for different graphs}
	Let $ \Gamma $ and $ \Gamma' $ be two graphs with the same number of edges that satisfy Assumption \ref{ass: assumptions} and consider their polynomials $ P_{\Gamma}$ and $ P_{\Gamma'}  $. Then,
	\begin{enumerate}
		\item The polynomials are equal $ P_{\Gamma}=P_{\Gamma'} $, equivalently $ \Sigma(\Gamma)=\Sigma(\Gamma') $, if and only if $ \Gamma $ and $ \Gamma' $ are isomorphic graphs.
		\item If $ \Gamma $ and $ \Gamma' $ are not isomorphic, then $ P_{\Gamma}$ and $ P_{\Gamma'} $ \textbf{do not share any common factor}, except for two cases:
		\begin{enumerate}
			\item If $ \Gamma $ is a mandarin and $ \Gamma' $ is a flower, or vice versa.
			\item If $ \Gamma $ and $ \Gamma' $ share a common loop $ e_{j} $, in which case $ (1-z_{j}) $ is a common factor of $ P_{\Gamma}$ and $ P_{\Gamma'} $.
		\end{enumerate}  
	\end{enumerate} 
\end{lem}
\begin{proof}
	Part (1). Two graphs $ \Gamma $ and $ \Gamma' $ of $ N $ edges are isomorphic if and only if $ \spec(\Gamma,\lv)=\spec(\Gamma',\lv) $ for any $ \lv $.  In fact, it is enough to consider only one $ \Q $-independent choice of $ \lv $ as seen in  \cite{GutSmi_jpa01,KurNow_jpa05}. According to Lemma \ref{lem: sec man in U}, having $ \spec(\Gamma,\lv)=\spec(\Gamma',\lv) $ for any $ \lv $ is equivalent to
	\[\Sigma(\Gamma):=Z(P_{\Gamma})\cap\T^{N}=Z(P_{\Gamma'})\cap\T^{N}=:\Sigma(\Gamma'),\]
	which, by Lemma \ref{lem: zariski}, is equivalent to 
	\[Z(P_{\Gamma})=Z(P_{\Gamma'}).\]
	According to Theorem \ref{Thm: Sarnak-Kurasov}, both $ P_{\Gamma} $ and $ P_{\Gamma'} $ have no square factors, so the equality of their zero sets implies that the polynomials are equal up to a constant. This constant is $ 1 $ since $ P_{\Gamma}(0)=P_{\Gamma'}(0)=1 $ by construction. We may conclude that $ P_{\Gamma}=P_{\Gamma'} $ if and only if $ \Gamma $ and $ \Gamma' $ are isomorphic. \\
	
	To prove (2), assume that $ P_{\Gamma}\ne P_{\Gamma'} $ and further assume they have a common factor $ q $. We may assume without loss of generality that $ P_{\Gamma} $ is reducible and that $ q $ is an irreducible factor of $ P_{\Gamma} $. As discussed in Remark \ref{rem: degree of PG}, both $ P_{\Gamma} $ and $ P_{\Gamma'} $ have degree 2 in each $ z_{j} $, so $ q $ has degree at most $ 1 $ in some variable $ z_{j} $, and hence $ q\ne c P_{\Gamma'} $ for any constant $ c\in\C $. Therefore, $ q $ must is a non-trivial factor of $ P_{\Gamma'} $, which means that $ P_{\Gamma'} $ is reducible. We may conclude that either both graphs have loops, or one of them, say $ \Gamma $ with out loss of generality, is a mandarin. We treat the two cases separately.\\

	Case (i), assume $ \Gamma $ is a mandarin graph and $ \Gamma' $ has loops. We want need to show that if their polynomials share a common factor $ q $, then $ \Gamma' $ is a flower. Since $ \Gamma $ is a mandarin, either $ q=P_{M,\mathrm{s}} $ or $ q=P_{M,\mathrm{as}} $, and in both cases $ q $ is irreducible and has degree 1 in every $ z_{j} $. Since $ \Gamma' $ is a graph with loops, having such an irreducible factor implies that all edges are loops, by Remark \ref{rem: degree of PG}. We conclude that $ \Gamma' $ must be a flower.\\

	Case (ii), assume both $ \Gamma $ and $ \Gamma' $ have loops and that their polynomials share a common factor $ q $. Assume by contradiction that they do not share a common loop. Then their common factor must be $ P_{\Gammasym}=P_{\Gamma',\mathrm{sym}} $ by Theorem \ref{Thm: Sarnak-Kurasov}. Let $ j $ such that $ e_{j} $ is a loop of $ \Gamma $ and not a loop of $ \Gamma' $, then $ P_{\Gammasym}$ has degree 1 in $ z_{j} $ but $ P_{\Gamma',\mathrm{sym}}  $ has degree $ 2 $ in $ z_{j} $, according to Remark \ref{rem: degree of PG}. This leads to the needed contradiction. Hence  $ \Gamma $ and $ \Gamma'  $ share a loop edge $ e_{j} $ and so $ (1-z_{j}) $ is a common factor.
	

\end{proof}     

\subsection{Reflection symmetries and the trace space}\label{sec: subsection Reflections}
	Recall that for a mandarin graph, an eigenfunction $ f $ is called \emph{symmetric} if $ f\circ R=f $, where $ R $ is the reflection of all edges. Similarly, $ f $ is called \emph{anti-symmetric} if $ f\circ R=-f $. It is not hard to see that these properties are determined by the traces. Given an eigenpair $ (k^2,f) $, $ f $ is symmetric if and only if $ \tr_{k}(f) $ has $ (A_{j},B_{j})=(C_{j},D_{j}) $ for every edge $ e_{j} $, and it is anti-symmetric if and only if $ (A_{j},B_{j})=-(C_{j},D_{j}) $ for all edges.

\begin{lem}[Mandarin trace space symmetry]\label{lem: Mandarin trace space}
	Let $ \Gamma $ be a mandarin graph and consider a trace fiber $ \TS_{\bz}(\Gamma) $ for $ \bz\in\Sigma(\Gamma) $.  Define its symmetric and anti-symmetric subspaces
	\begin{align*}
		\TS_{\bz,\mathrm{s}}(\Gamma)&:=\set{\xv\in\TS_{\bz}(\Gamma)}{(A_{j},B_{j})=(C_{j},D_{j})~~~\mbox{for every edge   }e_{j}},\quad\mbox{and}\\
		\TS_{\bz,\mathrm{as}}(\Gamma)&:=\set{\xv\in\TS_{\bz}(\Gamma)}{(A_{j},B_{j})=-(C_{j},D_{j})~~~\mbox{for every edge   }e_{j}}.
	\end{align*}
Then, 
	\[\TS_{\bz}(\Gamma)=\TS_{\bz,\mathrm{s}}(\Gamma)\oplus\TS_{\bz,\mathrm{as}}(\Gamma),\]
	with the rule of
	\begin{align*}
		\TS_{\bz,\mathrm{s}}(\Gamma)\ne\{0\}&\iff P_{M,\mathrm{s}}(\bz)=0\\
		\TS_{\bz,\mathrm{as}}(\Gamma)\ne\{0\}&\iff P_{M,\mathrm{as}}(\bz)=0.
	\end{align*} 	 
\end{lem}
For later use, we define the symmetric and anti-symmetric parts of the trace space.
\begin{defn}\label{def: mandarin TSs TSas}
	Given a mandarin graph $ \Gamma $, define
		\begin{align*}
		\TS_{\mathrm{s}}(\Gamma)&:=\set{(\bz,\xv)\in\TS(\Gamma)}{P_{M,\mathrm{s}}(\bz)=0,\quad\mbox{ and    }\xv\in \TS_{\bz,\mathrm{s}}(\Gamma)},\quad\mbox{and}\\
		\TS_{\mathrm{as}}(\Gamma)&:=\set{(\bz,\xv)\in\TS(\Gamma)}{P_{M,\mathrm{s}}(\bz)=0,\quad\mbox{ and    }\xv\in \TS_{\bz,\mathrm{as}}(\Gamma)}.
	\end{align*}
In other words, given an eigenpair $ (k^2,f) $ of $ (\Gamma,\lv) $,
\begin{enumerate}
	\item $ f $ is symmetric if and only $ (\exp(ik\lv),\tr_{k}(f))\in \TS_{\mathrm{s}}(\Gamma)$.
		\item $ f $ is anti-symmetric if and only $ (\exp(ik\lv),\tr_{k}(f))\in \TS_{\mathrm{as}}(\Gamma)$.
\end{enumerate}  
\end{defn}
Prior to proving Lemma \ref{lem: Mandarin trace space}, we first state the analogous result for graphs with loops. It is straight forward that if $ (k^2,f) $ is an eigenpair and $ f $ is supported on a loop $ e_{j} $, then $\tr_{k}(f) $ vanishes on all edges except for $ e_{j} $, and the vertex condition at the vertex of $ e_{j} $ is
\[A_{j}=C_{j}=0,\qquad B_{j}+D_{j}=0.\] 
It is not hard to see that the other direction holds too, i.e., $ f $ is supported on $ e_{j} $ if $\tr_{k}(f) $ has $ B_{j}=-D_{j} $ and vanish on all other entries.
\begin{lem}[Trace space symmetry for loops]\label{lem: trace space for loops}
	Let $ \Gamma $ be a graph with loops, let $ \EL $ be the set of loops, and consider a trace fiber $ \TS_{\bz}(\Gamma) $ for $ \bz\in\Sigma(\Gamma) $. For any loop $ e_{j}\in\EL $, define the anti-symmetric subspace
	\begin{equation*}
		\TS_{\bz,\mathrm{as},j}(\Gamma):=\set{\xv\in\TS_{\bz}(\Gamma)}{D_{j}=-B_{j},~~\mbox{and all other entries of   }\xv\mbox{  vanish}}.
	\end{equation*}
Define the symmetric (on all loops) subspace
	\begin{equation*}
		\TS_{\bz,\mathrm{sym}}(\Gamma):=\set{\xv\in\TS_{\bz}(\Gamma)}{(A_{j},B_{j})=(C_{j},D_{j})\quad\mbox{for every loop   }e_{j}\in\EL}.
	\end{equation*}
	Then, 
	\[\TS_{\bz}(\Gamma)=\TS_{\bz,\mathrm{sym}}(\Gamma)\bigoplus_{e_j\in\EL}\TS_{\bz,\mathrm{as},j}(\Gamma),\]
	with the rule of
	\begin{align*}
		\TS_{\bz,\mathrm{sym}}(\Gamma)\ne\{0\}&\iff P_{\Gammasym}(\bz)=0\\
		\TS_{\bz,\mathrm{as},j}(\Gamma)\ne\{0\}&\iff z_{j}=1.
	\end{align*} 	 
\end{lem}
For later use, we define the symmetric and anti-symmetric parts of the trace space.
\begin{defn}\label{def: TSs TSas for loops}
	Given a graph $ \Gamma $ with loops $ \EL $, define
	\begin{align*}
		\TS_{\mathrm{sym}}(\Gamma)&:=\set{(\bz,\xv)\in\TS(\Gamma)}{P_{\Gammasym}(\bz)=0,\quad\mbox{ and    }\xv\in \TS_{\bz,\mathrm{sym}}(\Gamma)},\quad\mbox{and}\\
		\TS_{\mathrm{as},j}(\Gamma)&:=\set{(\bz,\xv)\in\TS(\Gamma)}{z_{j}=1,\quad\mbox{ and    }\xv\in \TS_{\bz,\mathrm{as},j}(\Gamma)}.
	\end{align*}
In particular, given an eigenpair $ (k^2,f) $ of $ (\Gamma,\lv) $, $ f $ is supported on a loop $ e_{j} $ if and only if $ (\exp(ik\lv),\tr_{k}(f))\in \TS_{\mathrm{as},j}(\Gamma)$.
\end{defn}
We may now prove these lemmas. 
\begin{proof}[Proof of Lemma \ref{lem: Mandarin trace space}]
	Let $ \Gamma $ be a mandarin graph, let $ \bz\in\Sigma(\Gamma) $ and let $ \lv\in\opcl{0,2\pi}^{N} $ such that $ \bz=\exp(ik\lv) $ with $ k=1 $. It is a standard argument that given an isometry $ R:(\Gamma,\lv)\to(\Gamma,\lv) $, such that $ R^{2} $ is the identity, any eigenspace has a basis of eigenfunctions that are either symmetric or anti-symmetric. See \cite{BanBerWey_jmp15} for example. Let $ \eig(\Gamma,\lv,k)_{\mathrm{s}} $ be the span of the symmetric basis eigenfunctions and $ \eig(\Gamma,\lv,k)_{\mathrm{as}} $ of the anti-symmetric basis eigenfunctions. Then
	\[\eig(\Gamma,\lv,k)=\eig(\Gamma,\lv,k)_{\mathrm{s}}\oplus\eig(\Gamma,\lv,k)_{\mathrm{as}}.\]
	Acting with $ \tr_{k} $ on this equation, using $ \bz=\exp(ik\lv) $, we get
	\[\TS_{\bz}(\Gamma)=\tr_{k}(\eig(\Gamma,\lv,k)_{\mathrm{s}})\oplus\tr_{k}(\eig(\Gamma,\lv,k)_{\mathrm{as}})=\TS_{\bz,\mathrm{s}}(\Gamma)\oplus\TS_{\bz,\mathrm{as}}(\Gamma).\]   
	Now let $ \xv\in\TS_{\bz,\mathrm{s}}(\Gamma) $ and we may assume that $ \xv $ is real. To be in $ \TS_{\bz,\mathrm{s}}(\Gamma) $, $ \xv $ must satisfy the vertex and edge conditions in \ref{lem: trace space polynomial equation} and to have $ (A_{j},B_{j})=(C_{j},D_{j}) $ on every edge $ e_{j} $. This can be reduced to the following conditions: 
	\begin{align}\label{eq: A equal Aj}
		A_{1}=A_{2}=\ldots=A_{N}& =:A,\\
		B_{1}+B_{2}+\ldots+B_{N}& =0,\label{eq: sum of B}
	\end{align}
and for every $ e_{j} $,
\begin{equation}\label{eq: zj in Aj and Bj}
	z_{j}=\frac{A+iB_{j}}{A-iB_{j}},\quad\mbox{or}\quad \|(A,B_{j}\|=0.
\end{equation}
We now consider two cases.\\

The first case is when $ A\ne0 $, in which case $ z_{j}\ne -1 $ for all $ j $, by \eqref{eq: zj in Aj and Bj}. Inverting the Möbius transformation gives $ B_{j}=-iA\frac{z_{j}-1}{z_{j}+1} $, and \eqref{eq: sum of B} leads to
\[\sum_{j=1}^{N}\frac{z_{j}-1}{z_{j}+1}=0, \quad\mbox{and therefore}\quad P_{M,\mathrm{s}}(\bz)=\left(\sum_{j=1}^{N}\frac{z_{j}-1}{z_{j}+1}\right)\prod_{j=1}^{N}(z_{j}+1)=0.\]
For the other direction, assume that $ P_{M,\mathrm{s}}(\bz)=0 $ with $ z_{j}\ne -1 $ for all $ j $. If we set $ \xv $ to have $ A=1 $ and $ B_{j}=-i\frac{z_{j}-1}{z_{j}+1} $ for all $ e_{j} $, then $ \xv\in \TS_{\bz,\mathrm{s}} $.\\

The second case is when $ A=0 $. If $ A=0 $, for every $ j $ either $ z_{j}=-1 $ or $ B_{j}=0 0 $, by \eqref{eq: zj in Aj and Bj}. The sum in  \ref{eq: sum of B} requires that there are at least two non zero $ B_{j} $'s, and therefore two coordinates satisfying $ z_{j}=-1 $. Hence, $ P_{M,\mathrm{s}}(\bz)=0 $. For the other direction, assume that $ P_{M,\mathrm{s}}(\bz)=0 $ and that $ z_{j}=-1 $ for some fixed $ e_{j} $ (as the case where all $ z_{j}\ne-1 $ was treated already). Then the equation
\[P_{M,\mathrm{s}}(\bz)=(z_{j}-1)\prod_{j'\ne j}(z_{j'}+1)=0,\]
tells us that there must be another $ j'\ne j $ with $ z_{j'}=-1 $. Let $ \xv $ be a trace vector with
\[B_{j}=D_{j}=-B_{j'}=-D_{j'},\] and zero in all other entries, then $ \xv\in \TS_{\bz,\mathrm{s}} $.\\

We conclude that 
\[\TS_{\bz,\mathrm{s}}\ne\{0\}\iff P_{M,\mathrm{s}}(\bz).\]
The proof of
\[\TS_{\bz,\mathrm{as}}\ne\{0\}\iff P_{M,\mathrm{as}}(\bz),\]
follows the same steps.  
\end{proof}
The proof of Lemma \ref{lem: trace space for loops} is similar to the proof of Lemma \ref{lem: Mandarin trace space}. 
\begin{proof}[Proof of Lemma \ref{lem: trace space for loops}]
	Let $ \Gamma $ be a graph with loops, let $ \EL $ be the set of loops and for every $ e_{j}\in\EL $ let $ R_{j} $ be the reflection of $ e_{j} $ that acts by identity on all other edges. As already mentioned, the group of reflection symmetries of $ \Gamma $, say $ G $, is the group generated by $ R_{j} $ for all $ e_{j}\in\EL $. Notice that this is an abelian group since the $ R_{j} $ generators commute, and that any element of this group satisfies $ R^{2} $ equals identity. The same argument as before, namely that these are isometries that preserve the vertex conditions, tells us that any eigenspace $ \eig(\Gamma,\lv,k) $ has a basis of eigenfunctions that satisfy $ f\circ R_{j}=\pm f $ for every loop $ e_{j} $. Notice that there are only two cases of an eigenfunction $ f $ as above: 
	\begin{enumerate}
		\item Either $ f\circ R_{j}=f $ for every loop $ e_{j} $, in which case we call $ f $ \emph{symmetric}. Or,
		\item $ f $ is supported on some loop $ e_{j} $, in which case $ f\circ R_{j}=-f $.
	\end{enumerate} 
To see that these are the only two cases, assume that $ f $ is not supported on a single loop, but $ f\circ R_{j}=-f $ for some loop $ e_{j} $. Then there exists an edge $ e_{j'} $, with $ j'\ne j $ such that $ f|_{e_{j'}}\not\equiv 0 $. But $ R_{j} $ acts as identity on $ e_{j'} $ and so $ f\circ R_{j}=-f $ implies that $ f|_{e_{j'}}=-f|_{e_{j'}}$, contradicting the assumption of t $ f|_{e_{j'}}\not\equiv 0 $. Hence, the above two cases are indeed the only two cases. \\
As in the proof of Lemma \ref{lem: Mandarin trace space}, the basis of eigenfunction with certain types of symmetry provides the needed decomposition of any eigenspace, which results in
\begin{equation}\label{eq: loop partitoin}
	\TS_{\bz}(\Gamma)=\TS_{\bz,\mathrm{sym}}(\Gamma)\bigoplus_{e_j\in\EL}\TS_{\bz,\mathrm{as},j}(\Gamma)
\end{equation}
	Notice that by definition, any non-trivial $ \TS_{\bz,\mathrm{as},j}(\Gamma) $ is one dimensional. Assume that there is a non-zero trace vector $\xv\in \TS_{\bz,\mathrm{as},j}(\Gamma) $ for some loop $ e_{j}  $, namely $ B_{j}=-D_{j}\ne0 $ and the rest of the entries are zero. The edge equation in Lemma \ref{lem: trace space polynomial equation} gives,
	\[B_{j}=z_{j}B_{j}\Rightarrow z_{j}=1.\]
On the other hand, given any $ \bz\in\T^{N} $ with $ z_{j}=1 $ for some loop $ e_{j} $, the trace vector $ \xv $ with $ B_{j}=-D_{j}=1 $ and zero in all other entries, satisfies $\xv\in \TS_{\bz,\mathrm{as},j}(\Gamma) $. Therefore,
\begin{equation}\label{eq: loop condition}
	\TS_{\bz,\mathrm{as},j}(\Gamma)\ne\{0\}\iff z_{j}=1.
\end{equation}
We are left with showing that
\[\TS_{\bz,\mathrm{sym}}(\Gamma)\ne\{0\}\iff P_{\Gammasym}(\bz)=0.\]
To do so, we use a fixed orthogonal decomposition of $ \C^{2N} $,
\begin{equation}\label{eq: C2N decomposition}
	\C^{2N}=V_{\mathrm{sym}}\oplus V_{\mathrm{as}},
\end{equation}
that was constructed in \cite[Definition 5.11]{AloBanBer21universality}, where the following properties we shown:
\begin{enumerate}
	\item The $ \bz $ dependent matrix $ U(\bz) $ is block diagonal in the fixed decomposition \eqref{eq: C2N decomposition},
		\[U(\bz)=U_{\mathrm{sym}}(\bz)\oplus U_{\mathrm{as}}(\bz),\]
	for any $ \bz\in\T^{N} $. 
	\item The space $ V_{\mathrm{as}} $ is $ |\EL| $ dimensional, with a (fixed) basis of vectors $ \textbf{a}_{e_{j}} $ for $ e_{j}\in\EL $,  that satisfy
	\[U(\bz)\textbf{a}_{e_{j}}=z_{j}\textbf{a}_{e_{j}},\]
	for any $ \bz\in\T^{N} $.
\end{enumerate}
We may conclude that
\begin{align}
	P_{\Gamma}(\bz):= & \det(\mathbb{I}_{2N}-U_{\bz})\\
	= & \det(\mathbb{I}_{|\EL|}-U_{\mathrm{sym}}(\bz))\det(\mathbb{I}_{2N-|\EL|}-U_{\mathrm{as}}(\bz))\\
	=&\left(\prod_{e_{j}\in\EL}(1-z_{j})\right)\det(\mathbb{I}_{2N-|\EL|}-U_{\mathrm{as}}(\bz)),
\end{align}
where the second line follows from the decomposition in (1) and in the third line we replace the determinant $ \det(\mathbb{I}_{|\EL|}-U_{\mathrm{sym}}(\bz)) $ with the product of eigenvalues given by (2). By comparing this decomposition of $ P_{\Gamma}(\bz) $ with Theorem \ref{Thm: Sarnak-Kurasov} we conclude that 
\[P_{\Gammasym}(\bz)=\det(\mathbb{I}_{2N-|\EL|}-U_{\mathrm{as}}(\bz)),\]
and in particular,
\begin{equation}\label{eq: something}
	\dim(\ker(\mathbb{I}_{2N}-U(\bz)))>|\set{e_{j}\in\EL}{z_{j}=1}|\iff P_{\Gammasym}(\bz)=0.
\end{equation} 
We may now recall that 
\[\dim(\ker(\mathbb{I}_{2N}-U(\bz)))=\dim(\TS_{\bz}(\Gamma))=\dim(\TS_{\bz,\mathrm{sym}}(\Gamma))+\sum_{e_{j}\in\EL}\dim(\TS_{\bz,\mathrm{as},j}(\Gamma)),\]
by Corollary \ref{cor: isomorhpism} and the decomposition \eqref{eq: loop partitoin}. According to \eqref{eq: loop condition}, $ \dim(\TS_{\bz,\mathrm{as},j}(\Gamma)) $ equals one when $ z_{j}=1 $ and is zero otherwise. We conclude that 
\[\dim(\ker(\mathbb{I}_{2N}-U(\bz)))>|\set{e_{j}\in\EL}{z_{j}=1}|\iff \TS_{\bz,\mathrm{sym}}(\Gamma)\ne\{0\},\]
and we are done by applying \eqref{eq: something}.
\end{proof}

\begin{lem}[Mandarin and flower]\label{lem: mandarin and flower}
	If $ \Gamma $ is a flower graph and $ \Gamma' $ is a mandarin graph with the same number of edges, then
	\[\TS_{\mathrm{sym}}(\Gamma)=\TS_{\mathrm{s}}(\Gamma').\]
	In particular, $ P_{\Gammasym}=c P_{M,\mathrm{s}} $ for some constant $ c $.
\end{lem}
\begin{proof}
	A point $ (\bz,\xv)\in \TS_{\mathrm{s}}(\Gamma') $ needs to satisfy $ (A_{j},B_{j}=C_{j},D_{J}) $ for every $ e_{j} $ (since all edges are loops) and also the equations in Lemma \ref{lem: trace space polynomial equation}. Assuming that $ \xv $ is real\footnote{We can assume $ \xv $ is real because we know that every fiber $ \TS_{\bz}(\Gamma') $ has a real basis.}, it can be reduced to the following conditions: 
	\begin{align}\label{eq: A equal Aj}
		A_{1}=A_{2}=\ldots=A_{N}& =:A,\\
		2B_{1}+2B_{2}+\ldots+2B_{N}& =0,\label{eq: sum of B}
	\end{align}
	and for every $ e_{j} $,
	\begin{equation}\label{eq: zj in Aj and Bj}
		z_{j}=\frac{A+iB_{j}}{A-iB_{j}},\quad\mbox{or}\quad \|(A,B_{j}\|=0.
	\end{equation}
 These are exactly the same equations as \eqref{eq: A equal Aj}, \eqref{eq: sum of B} and \eqref{eq: zj in Aj and Bj}, which are the defining equations of $ \TS_{\mathrm{sym}}(\Gamma) $. We conclude that 
	\[\TS_{\mathrm{sym}}(\Gamma)=\TS_{\mathrm{s}}(\Gamma').\]
	According to Lemma \ref{lem: Mandarin trace space} and Lemma \ref{lem: trace space for loops} it means that for any $ \bz\in\T^{N} $,
	\[P_{\Gammasym}(\bz)=0\iff P_{M,\mathrm{s}}(\bz)=0.\]
	We may conclude from Lemma \ref{lem: zariski} that $ P_{\Gammasym}(\bz) $ and $ P_{M,\mathrm{s}} $ share the same zero set in $ \C^{N} $ and are therefore equal up to a constant factor.  
\end{proof}
	
	\section{Genericity theorems}
In this section we prove the main results of this paper. The proofs share a similar structure, in which we show that a certain property $ \prop $ is generic, by showing that any eigenpair $(k^2,f)  $ of $ (\Gamma,\lv) $ that fails to satisfy $ \prop $ must have $ \exp(ik\lv)\in B $ for some ``small" subvariety $ B\subset \Sigma(\Gamma) $. A common step in all proofs is the claim that a property whose negation is encapsulated by a small subvariety as above is strongly and ergodically generic. We prove this genericity criteria in the next Lemma. We remind the reader that we call $ V\subset\C^{N} $ an \emph{algebraic set} or a \emph{variety} if it is a finite union and intersection of zero sets of polynomials. We say that $ B\subset \Sigma(\Gamma) $ is a subvariety of $ \Sigma(\Gamma) $ if $ B=V\cap\Sigma(\Gamma) $ for some variety $ V $. A subvariety $ B $ has a positive co-dimension in $ \Sigma(\Gamma) $ if $ \dim(B)\le N-2 $, since $ \Sigma(\Gamma) $ has real dimension $ N-1 $ by Lemma \ref{lem: sec man in U}. 
	\begin{lem}[The genericity criteria]\label{lem: main genericty lem}
		Let $ \Gamma$ be a graph satisfying Assumption \ref{ass: assumptions}, and let $ B\subset\Sigma({\Gamma}) $ be a subvariety of positive co-dimension in $ \Sigma({\Gamma}) $. Then,
		\begin{enumerate}
			\item The set of ``good" lengths 
			\begin{equation}\label{eq: G(B)}
				G(B)=\set{\lv\in\R_{+}^E}{\forall k>0,\quad\exp(ik\lv)\notin B},
			\end{equation}
					is strongly generic.
			\item For any $ \Q $-independent $ \lv $, 
			\[\lim_{T\to\infty}\frac{|\set{k\in\spec (\Gamma,\lv)\cap [0,T]}{\exp(ik\lv)\in B}|}{|\spec(\Gamma,\lv)\cap[0,T]|}=0.\]
		\end{enumerate}
	\end{lem} 
	\begin{proof}
		By definition, to show that $ G(B) $ is strongly generic we need to show that its complement $ G(B)^{c}=\R_{+}^{N}\setminus G(B)$ is a countable union of sets, $ \mathcal{B}_{n} $ for $ n\in\N $, such that each $ \mathcal{B}_{n} $ is a closed subanalytic set (see Definition \ref{def: subanalytic}) of dimension at most $ N-1 $. We may deduce from \eqref{eq: G(B)} that when a subvariety is given by a union $ B=B_{1}\cup B_{2} $, the complement of $ G(B)=G(B_{1}\cup B_{2}) $ is given by a union,
		\[G(B_{1}\cup B_{2})^{c}=G(B_{1})^{c}\cup G(B_{2})^{c},\]
		and so $ G(B_{1}\cup B_{2}) $ is strongly generic if both $ G(B_{1}) $ and $ G(B_{2}) $ are strongly generic. It is therefore enough to prove that $ G(B) $ is strongly generic when the subvariety $ B$ is defined in terms of a finite intersection of zero sets of polynomial, rather than union and intersection. Assume that 
				\[B=\set{\bz\in\T^{N}}{p_{j}(\bz)=0,\quad j=0,1,2,\ldots,m },\]
where $ p_{1},p_{2},\ldots,p_{m} $ are polynomials. The real and imaginary parts of each polynomial $ p_{j} $ defines real analytic functions on $ \R\times\R^{N} $ by
\[f_{j,1}(k,\xv):=\Re[p_{j}(\exp{(ik\xv)})],\quad\mbox{and}\quad f_{j,2}(k,\xv):=\Im[p_{j}(\exp{(ik\xv)})].\] 
Given $ n\in\N $, define  
\begin{equation*}
	\mathcal{B}_{n}:=\set{\lv\in\R_{+}^N}{\exists k\in[\frac{1}{n},n]~~\mbox{ s.t.   }f_{j}(k,\lv)=0,\quad  s=1,2,\mbox{  and  }j=0,1,2,\ldots ,m},
\end{equation*}
which is a closed subanalytic set according to Definition \ref{def: subanalytic}. We write $ G(B)^{c} $ as
\begin{align*}
	G(B)^{c}:= & \set{\lv\in\R_{+}^N}{\exists k>0~~\mbox{ s.t.   }\exp(ik\lv)\in B}\\
	= & \set{\lv\in\R_{+}^N}{\exists k>0~~\mbox{ s.t.   }f_{j,s}(k,\lv)=0,\quad s=1,2,\mbox{  and  }j=0,1,2,\ldots ,m}\\
	= & \bigcup_{n\in\N}\mathcal{B}_{n}.
\end{align*}
To conclude that $ G(B) $ is strongly generic we need to show that $ \dim(\mathcal{B}_{n})\le N-1 $ for all $ n\in\N $. To this end, define the auxiliary sets 
		\[\tilde{B}:=\set{\xv\in\R^{N}}{\exp{(i\xv)}\in B},\quad C_{n}:=\set{(k,\lv)\in[\frac{1}{n},n]\times\R_{+}^{N}}{k\lv\in\tilde{B}}.\]
The exponent map $ e(\xv):=\exp{(i\xv)} $ is a local diffeomorhpism between $ \R^{N} $ and $ \T^{N} $, so \[\dim(\tilde{B})=\dim(B)\le N-2,\]
follows from $ e(\tilde{B})=B $ and the assumption that $ B $ has positive co-dimension in $ \Sigma(\Gamma) $. The dimension of $ C_{n} $ is bounded by
		\[\dim(C_{n})\le\dim(\tilde{B})+1\le N-1,\]
		which bounds the dimension of $ \mathcal{B}_{n} $ by 
		\[\dim(\mathcal{B}_{n})\le\dim(C_{n})\le N-1,\]
since $ \mathcal{B}_{n} $ is a projection of $ C_{n} $. This proves (1).\\
		
		To prove (2), consider the embedding of $ \Sigma(\Gamma) $ in the flat torus $ \R^{N}/2\pi\Z^{N} $,
		\[\mflat(\Gamma):=\set{\xv\in\R^{N}/2\pi\Z^{N}}{\exp(i\xv)\in\Sigma(\Gamma)}.\]
		This is an analytic variety, defined by $ P_{\Gamma}(\exp(i\xv))=0 $, that has dimension $ N-1 $ (as it is diffeomorhpic to $ \Sigma(\Gamma) $ ). We abuse notation and consider the periodic set $ \tilde{B} $ as a subset of $ \R^{N}/2\pi\Z^{N} $, and therefore a subset $ \tilde{B}\subset \mflat(\Gamma) $. Let $ \fr{k\lv} $ denote the reminder of $ k\lv $ modulo $ 2\pi $. In \cite{BarGas_jsp00} Barra and Gaspard introduced an $ \lv $ dependent Borel measure $ \mu_{\lv} $ on $ \mflat(\Gamma) $ which has the following ergodic property. For any $ \Q $-independent $ \lv $ and any subset $ A\subset \mflat(\Gamma)$, assuming its boundary satisfy $ \mu_{\lv}(\partial A)=0 $,
		\[\lim_{T\to\infty} \frac{|\set{k\in\spec(\Gamma,\lv)\cap[0,T]}{\fr{k\lv}\in A}|}{|\spec(\Gamma,\lv)\cap[0,T]|}=\frac{\mu_{\lv}(A)}{\mu_{\lv}(\Sigma(\Gamma))}.\] 
For a proof, see \cite[Proposition 4.4]{BerWin_tams10} or \cite[Lemma 3.2 ]{CdV_ahp15}. The measure $ \mu_{\lv} $ is absolutely continuous with respect to an $ (N-1) $-dimensional volume measure on $\mflat(\Gamma)  $ and therefore $ \mu_{\lv}(\tilde{B})=\mu_{\lv}(\partial\tilde{B})=0 $. Here we use the fact that $ \tilde{B} $ is closed and has $ \dim(\tilde{B})\le N-2 $. Applying the ergodic property to $ \tilde{B}\subset \mflat(\Gamma) $ gives
		\[\lim_{T\to\infty} \frac{|\set{k\in\spec(\Gamma,\lv)\cap[0,T]}{\fr{k\lv}\in \tilde{B}}|}{|\spec(\Gamma,\lv)\cap[0,T]|}=\frac{\mu_{\lv}(\tilde{B})}{\mu_{\lv}(\Sigma(\Gamma))}=0.\] 
	 
			\end{proof}
	At this point we get, as a corollary of Lemma \ref{lem: main genericty lem} and Corollary \ref{cor: singular set}, an independent proof for (a stronger version of) Friedlander's result on the simplicity of the spectrum. 
	\begin{cor}
		Let $ \Gamma $ be a graph satisfying Assumption \ref{ass: assumptions}. Then having simple eigenvalues is strongly and ergodically generic in $ \lv $.  
	\end{cor} 
\begin{proof}
By Lemma \ref{lem: sec man in U}, $ k^2>0 $ is a multiple eigenvalue of $ (\Gamma,\lv) $ if and only if $ \exp(ik\lv)\in\msing(\Gamma) $. By Corollary \ref{cor: singular set}, $ \msing(\Gamma) $  is a subvariety of positive codimension in $ \Sigma(\Gamma) $, and the needed result follows by substituting $ B=\msing $ in Lemma \ref{lem: main genericty lem}.
\end{proof}
\begin{rem}
	 The above proof is independent of Friedlander's proof in \cite{Fri_ijm05}. The idea of an alternative proof for the generic simplicity which relies on the positive codimension of $ \msing(\Gamma) $ appeared in section 7 of \cite{CdV_ahp15}. 
\end{rem}

\subsection{Proofs of Theorems \ref{thm: polynomial vertex conditions} and \ref{thm: polynomial vertex conditions mandarin} - Genericity on the trace space}$ ~ $\\
Theorem \ref{thm: polynomial vertex conditions} and Theorem \ref{thm: polynomial vertex conditions mandarin} can be stated as one general theorem, using the results and definitions accumulated so far.
\begin{thm}\label{thm: unified}
	Let $ \Gamma $ be a graph satisfying Assumption \ref{ass: assumptions}. Let $ p=\PG $ if $ \PG $ is irreducible, otherwise, let $ p $ be an irreducible factor of $ \PG $. Let $ q(\bz,\xv) $ be a polynomial in $ (\bz,\xv) $ which is homogeneous in $ \xv $. If there exists a point $ (\bz,\xv)\in\TS(\Gamma) $ such that 
	\[\bz\in\mreg(\Gamma)\cap Z(p),\quad \xv\ne0,\quad \mbox{and}\quad q(\bz,\xv)\ne0.\]
	Then, the next two properties of eigenpairs $ (k^2,f) $ of $ (\Gamma,\lv) $ are strongly and ergodically generic in $ \lv $:
	\begin{enumerate}
		\item $ \exp(ik\lv)\in\mreg(\Gamma) $, and
		\item $ q(\exp(ik\lv),\tr_{k}(f))\ne0$ whenever $ p(\exp(ik\lv))=0 $.
	\end{enumerate}
\end{thm}
To see why Theorem \ref{thm: unified} implies Theorems \ref{thm: polynomial vertex conditions} and \ref{thm: polynomial vertex conditions mandarin}, let us break down these theorems into assumption and resulting generic property. The assumption in Theorems \ref{thm: polynomial vertex conditions} and \ref{thm: polynomial vertex conditions mandarin} is that there exists an $ \lv $ and an eigenpair $ (k^2,f) $ of $ (\Gamma,\lv) $ such that $ k^2>0 $ and is a simple eigenvalue, $ q(\exp(ik\lv),\tr_{k}(f))\ne0$, and $ f $ has a certain ``symmetry type". The result is that this is the generic situation for eigenfunctions of that ``symmetry type". Namely, for a generic eigenpair $ (k^2,f) $, $ k^2>0 $ is simple and $ q(\exp(ik\lv),\tr_{k}(f))\ne0$ whenever $ f $ is of that ``symmetry type". The four ``symmetry types" are
\begin{enumerate}
	\item Theorem \ref{thm: polynomial vertex conditions} for graphs with no loops: Any $ f $.
	\item Theorem \ref{thm: polynomial vertex conditions} for graphs with loops: $ f $ is not supported on a single loops.  
	\item Theorem \ref{thm: polynomial vertex conditions mandarin} for mandarin graphs, first case: $ f $ is symmetric.
	\item Theorem \ref{thm: polynomial vertex conditions mandarin} for mandarin graphs, second case: $ f $ is anti-symmetric.
\end{enumerate}
As shown in Lemma \ref{lem: Mandarin trace space} and Lemma \ref{lem: trace space for loops}, the possible ``symmetry types" of $ f $ are captured by the decomposition of $ \dim(\TS_{\bz}) $ for $ \bz=\exp(ik\lv) $. In the case that $ k^2>0 $ is simple, equivalently $ \dim(\TS_{\bz})=1 $, there is only one possible ``symmetry type" and it is determined by the irreducible factor of $ \PG $ that vanish at $ \bz=\exp(ik\lv) $. We call this factor $ p $. 
\begin{enumerate}
	\item For Theorem \ref{thm: polynomial vertex conditions}, for graphs with no loops, $ p=\PG $.
	\item For Theorem \ref{thm: polynomial vertex conditions}, for graphs with loops, $ p=P_{\Gammasym} $.
	\item For Theorem \ref{thm: polynomial vertex conditions mandarin}, for symmetric eigenfunctions of mandarin graphs, $ p=P_{M,\mathrm{s}} $
	\item For Theorem \ref{thm: polynomial vertex conditions mandarin}, for anti-symmetric eigenfunctions of mandarin graphs, $ p=P_{M,\mathrm{as}} $
\end{enumerate}
Using the above dictionary it is a simple check to see that indeed Theorem \ref{thm: unified} implies Theorems \ref{thm: polynomial vertex conditions} and \ref{thm: polynomial vertex conditions mandarin}. We proceed with the proof of Theorem \ref{thm: unified}.  
\begin{proof}[Proof of theorem \ref{thm: unified}] We first state the following claim.\\
	\textbf{Claim:} There exist $ 4N $ polynomials $ Q_{j}\in\C[z_{1},\ldots,z_{n}]$ for $ j=1,2,\ldots,4N  $, such that for any $ (\bz,\xv)\in\TS(\Gamma) $ with $ \bz\in\mreg(\Gamma) $ and $ \xv\ne 0 $,
	\begin{equation}\label{eq: claim}
		q(\bz,\xv)=0\iff \bz\in Q_{j}(\bz)=0 \quad\mbox{for}\quad j=1,2,\ldots,4N.
	\end{equation}
We will first prove Theorem \ref{thm: unified} assuming the claim and then prove the claim. Define the variety $ V $ as the common zero set
\[V:=\set{\bz\in\C^{N}}{Q_{j}(\bz)=0 \quad\mbox{for}\quad j=1,2,\ldots,4N},\]
and the associated subvariety $ B\subset\Sigma(\Gamma)  $ by 
\[B=V\cap Z(p)\cap\T^{N}\cup\msing(\Gamma).\]
The subvariety $ B $ captures the negation of the generic properties (1) and (2) in Theorem \ref{thm: unified}. Clearly, if (1) fails, then $ \exp(ik\lv)\in\msing(\Gamma)\subset B  $. If (2) fails but not (1), then the point $ (\bz,\xv)=(\exp(ik\lv),\tr_{k}(f))\in\TS(\Gamma) $ has $ \bz\in\mreg $, $ \xv\ne 0 $, $ p(\bz)=0 $ and $ q(\bz,\xv)=0 $. Then $ \bz\in V $ according to the claim and so
\[\bz=\exp(ik\lv)\in V\cap Z(p)\cap\T^{N}\subset B.\] 
Applying Lemma \ref{lem: main genericty lem} to the subvariety $ B $ proves Theorem \ref{thm: unified}. To this end, we only need to show that $ B $ has positive co-dimension in $ \Sigma(\Gamma) $. In fact, it is enough to show that 
\begin{equation}\label{eq: dim V cap bla le N-2}
	\dim(V\cap Z(p)\cap\T^{N})\le N-2,
\end{equation}
since $ \dim(\msing(\Gamma))\le N-2 $ by Corollary \ref{cor: singular set}. By the assumption of Theorem \ref{thm: unified}, there exists a point $ (\bz,\xv)=(\exp(ik\lv),\tr_{k}(f))\in\TS(\Gamma) $ such that $ \bz\in\mreg $, $ \xv\ne 0 $, $ p(\bz)=0 $ and $ q(\bz,\xv)\ne0 $. According to the claim it means that 
\begin{equation}\label{eq: assumption and claim}
	\bz\in\cap Z(p)\setminus V, 
\end{equation}
so $ p $ (which is irreducible by our choice) is not a factor of at least one $ Q_{j} $ polynomial. We prove \eqref{eq: dim V cap bla le N-2} by applying Lemma \ref{lem: zariski} to $ p $ and this $ Q_{j} $ which gives  
\[\dim(V\cap Z(p)\cap\T^{N})\le\dim(Z(Q_{j})\cap Z(p)\cap\T^{N})\le N-2.\]
We conclude that $ \dim(B)\le N-2 $ which proves Theorem \ref{thm: unified} by Lemma \ref{lem: main genericty lem}.\\

We now prove the claim on which our proof is based. Write $ q(\bz,\xv) $ as the sum of $ K $ monomials, using multi-indices $ \textbf{a}_{n}\in(\N\cup\{0\})^{N} $ and $ \textbf{b}_{n}\in(\N\cup\{0\})^{4N} $ for $ n=1,\ldots,K $,
\[q(\bz,\xv)=\sum_{n=1}^{K}\bz^{a_{n}}\xv^{b_{n}}, \qquad \bz^{a_{n}}:=\prod_{j=1}^{N}z_{j}^{a_{n}(j)},\quad \xv^{a_{n}}:=\prod_{j=1}^{4N}x_{j}^{b_{n}(j)}\]
Recall that $ q(\bz,\xv) $ is homogeneous in $ \xv $, so there is some $ m\in\N\cup\{0\} $ such that 
\[|\textbf{b}_{n}|:= \sum_{j=1}^{4N}\textbf{b}_{n}(j)=m, \quad\mbox{for all}\quad n=1,2,\ldots K.\]
Consider the rank one matrix $ \xv\xv^{*} $ whose entries are $(\xv\xv^{*})_{i,j}= x_{i}\overline{x}_{j} $. The following holds
\[\overline{x}_{j}^{m}q(\bz,\xv)=\sum_{n=1}^{K}\bz^{a_{n}}\prod_{i=1}^{4N}(\xv\xv^{*})_{b_{n}(i),j}\quad\mbox{for all}\quad j=1,2,\ldots,4N.\]
Let $ A(\bz) $ be the $ 4N\times 4N $ matrix introduced in Lemma \ref{lem: trace as Az}, and define the polynomials,
\[Q_{j}(\bz):=\sum_{n=1}^{K}\bz^{a_{n}}\prod_{i=1}^{4N}(A(\bz))_{b_{n}(i),j} \quad\mbox{for all}\quad j=1,2,\ldots,4N.\]
These are indeed polynomials since the entries of $ A(\bz) $ are polynomials, by Lemma \ref{lem: trace as Az}. Fix a point $ (\bz,\xv)\in\TS(\Gamma) $ with $ \bz\in\mreg(\Gamma) $ and $ \xv\ne 0 $. According to Lemma \ref{lem: trace as Az}, 
\[A(\bz)=c_{\bz,\xv}\xv\xv^{*},\]
for some non-zero constant $ c_{\bz,\xv}\in\C\setminus\{0\} $. Therefore,
\[Q_{j}(\bz)=c_{\bz,\xv}^{m}(\overline{x_{j}})^{m}q(\bz,\xv) \quad\mbox{for every}\quad j=1,2,\ldots,4N.\]
If $ q(\bz,\xv)=0 $ then $ Q_{j}(\bz)=0 $ for all $ j $. For the other direction, assume that $ Q_{j}(\bz)=0 $ for all $ j $. Since $ \xv\ne 0 $ then $ x_{j}\ne0 $ for some $ j $, in which case $ Q_{j}(\bz)=0 $ implies $ q(\bz,\xv)=0 $. We conclude that for any point $ (\bz,\xv)\in\TS(\Gamma) $ with $ \bz\in\mreg(\Gamma) $ and $ \xv\ne 0 $,  
 \begin{equation*}
	q(\bz,\xv)=0\iff  Q_{j}(\bz)=0\quad\mbox{for every}\quad j=1,2,\ldots,4N.
\end{equation*}
\end{proof}
\subsection{Proof of Theorem \ref{thm: disjoint spectrum} - No common spectrum}
We remind the reader that Theorem \ref{thm: disjoint spectrum} considers the common spectrum, 
\[\spec(\Gamma,\lv)\cap\spec(\Gamma',\lv),\]
of two distinct\footnote{By distinct we mean non isomorphic.} graphs $ \Gamma $ and $ \Gamma' $ of same number of edges, assigned with the same edge lengths $ \lv=\lv' $. The theorem states that except for two cases, generically, there are no common eigenvalues. The two exceptional cases are: 
\begin{enumerate}
	\item[i)] If $ \Gamma $ and $ \Gamma' $ share a common  a common loops $ e_{j} $, then for any $ \lv=\lv' $,
	\[\frac{2\pi}{\ell_{j}}\N\subset\spec(\Gamma,\lv)\cap\spec(\Gamma',\lv), \]
		which means that the common spectrum has positive density,
	\[\liminf_{T\to\infty}\frac{|\spec(\Gamma,\lv)\cap\spec(\Gamma',\lv)\cap[0,T]|}{|\spec(\Gamma,\lv)\cap[0,T]|}\ge\frac{2L}{\ell_{j}},\qquad  L=\sum_{j=1}^{N}\ell_{j}.\]
	\item [ii)] If $ \Gamma $ is a mandarin graph and  $ \Gamma' $ is a flower graph, then for any $ \lv=\lv' $, the common spectrum is at least half of the spectrum, i.e.,
	\[\liminf_{T\to\infty}\frac{|\spec(\Gamma,\lv)\cap\spec(\Gamma',\lv)\cap[0,T]|}{|\spec(\Gamma,\lv)\cap[0,T]|}\ge\frac{1}{2}.\]
\end{enumerate}
Theorem \ref{thm: disjoint spectrum} can now follow from Theorem \ref{thm: unified} for $ \Gamma $ and $ q(\bz,\xv)=P_{\Gamma'} $, however, we will prove it using Lemma \ref{lem: main genericty lem} which was the main ingredient in the proof of Theorem \ref{thm: unified}.
\begin{proof}[Proof of Theorem \ref{thm: disjoint spectrum}]
Assume that $ \Gamma $ and $ \Gamma' $ are both graphs of $ N $ edges that satisfy Assumption \ref{ass: assumptions}. Recall that for any $ k\ge0 $, denoting $ \bz=\exp(ik\lv) $, we have 
\[k\in\spec(\Gamma,\lv)\cap\spec(\Gamma',\lv)\iff P_{\Gamma}(\bz)=0\quad\mbox{and}\quad P_{\Gamma'}(\bz)=0,\]
and define  \[B:=\Sigma(\Gamma)\cap\Sigma(\Gamma')=\set{\bz\in\T^{N}}{P_{\Gamma}(\bz)=0\quad\mbox{and}\quad P_{\Gamma'}(\bz)=0}.\]
Assume that $ \Gamma $ and $ \Gamma' $ are distinct, do not share a loop edge, and are not a pair of mandarin graph and flower graph. Then, the polynomials $ P_{\Gamma} $ and $ P_{\Gamma'} $ do not share any common factor, according to Lemma \ref{lem: no common factors for different graphs}, which means that $ B $ has positive co-dimension in $ \Sigma(\Gamma) $, by Lemma \ref{lem: zariski}. We conclude that $ B $ is a subvariety of $ \Sigma(\Gamma) $ that has positive co-dimension, so Lemma \ref{lem: main genericty lem} applies and the following holds: 
\begin{enumerate}
	\item The set of ``good" lengths 
	\begin{align*}
		G(B)= &\set{\lv\in\R_{+}^E}{\forall k>0,\quad\exp(ik\lv)\notin B}\\
		= & \set{\lv\in\R_{+}^E}{\spec(\Gamma,\lv)\cap\spec(\Gamma',\lv)=\{0\}},
	\end{align*}
	is strongly generic.
	\item For any $ \Q $-independent $ \lv $,
	\[\frac{|\set{k\in\spec (\Gamma,\lv)\cap [0,T]}{\exp(ik\lv)\in B}|}{|\spec(\Gamma,\lv)\cap[0,T]|}=\frac{|\spec (\Gamma,\lv)\cap\spec (\Gamma',\lv)\cap [0,T]|}{|\spec(\Gamma,\lv)\cap[0,T]|}\xrightarrow[T \to \infty]{} 0.\]
\end{enumerate}
This proves Theorem \ref{thm: disjoint spectrum}, except for the two special cases. \\

\textbf{Case i:} Assume that $ \Gamma $ and $ \Gamma' $ share a common loop $ e_{j} $, then according to Theorem \ref{Thm: Sarnak-Kurasov}, $ P_{\Gamma} $ and $ P_{\Gamma'} $ share a common factor $ (z_{j}-1) $, and so 
\[e^{ik\ell_{j}}=1\quad\Rightarrow\quad P_{\Gamma}(\exp(ik\lv))=0\quad\mbox{and}\quad P_{\Gamma'}(\exp(ik\lv))=0.\]
We conclude that $ k=\frac{2\pi}{\ell_{j}}n \in \spec (\Gamma,\lv)\cap\spec (\Gamma',\lv) $ for every $ n\in\N $, as needed. For the density statement we can write it as
\[|\spec(\Gamma,\lv)\cap\spec(\Gamma',\lv)\cap[0,T]|\ge\frac{\ell_{j}}{2\pi}T+O(1),\qquad T\to\infty.\]
Using the Weyl law, as stated in \cite[p. 95]{BerKuc_graphs} for example,
	\[\frac{|\spec(\Gamma,\lv)\cap\spec(\Gamma',\lv)\cap[0,T]|}{|\spec(\Gamma,\lv)\cap[0,T]|}\ge\frac{\frac{\ell_{j}}{2\pi}T+O(1)}{\frac{L}{\pi}T+O(1)}=\frac{\ell_{j}}{2L}+O(\frac{1}{T}),\qquad T\to\infty.\]\\

\textbf{Case ii:} Assume that $ \Gamma $ is a flower with $ N $ edges, i.e, every edge is a loop. According to the decomposition in Theorem \ref{Thm: Sarnak-Kurasov}, and the argument of case i, we have 
\begin{align*}
	|\set{k\in[0,T]}{P_{\Gammasym}(\exp(ik\lv))=0}|= & |\spec(\Gamma,\lv)\cap[0,T]|-\sum_{j=1}^{N}|\frac{2\pi}{\ell_{j}}\N \cap[0,T]|\\
	= & \frac{L}{\pi}T-\sum_{j=1}^{N}\frac{\ell_{j}}{2\pi}T+O(1),\\
	= & \frac{L}{2\pi}T+O(1),
\end{align*}
where we count, as usual, such that zeros of $ P_{\Gammasym}(\exp(ik\lv)) $ are repeated according to their degree and eigenvalues are repeated according their multiplicity. We conclude that
\[	\frac{|\set{k\in[0,T]}{P_{\Gammasym}(\exp(ik\lv))=0}|}{|\spec(\Gamma,\lv)\cap[0,T]|}=\frac{1}{2}+O(\frac{1}{T}),\qquad T\to\infty.\]
Now let $ \Gamma' $ be a mandarin graph with $ N $ edges. According to Lemma \ref{lem: mandarin and flower},
\[\set{k\in[0,T]}{P_{\Gammasym}(\exp(ik\lv))=0}\subset\spec(\Gamma,\lv)\cap\spec(\Gamma',\lv),\]
which finishes the proof. 
\end{proof}
 
\subsection{Proof of Theorem \ref{thm: non vanishing} - Non vanishing trace} Given a graph $ \Gamma $ with $ N $ edges, and $ \tilde{m} $ vertices of degree one, let $ m:=4N-m $, and number the entries of the associated trace vectors $ \xv=\tr_{k}(f) $ by
\[\xv=(x_{1},x_{2},\ldots,x_{m},0,0,\ldots,0),\]
such that the last $ \tilde{m}=4N-m $ entries are the Neumann entries (namely $ B_{j} $ or $ D_{j} $) corresponding to a vertex of degree one, and are therefore zero. In this way, Theorem \ref{thm: non vanishing} says that given a graph $ \Gamma $ that satisfy Assumption \ref{ass: assumptions}, the following properties of eigenpairs $ (k^2,f) $ of $ (\gamma,\lv) $ are strongly and ergodically generic in $ \lv $: 
\begin{enumerate}
	\item $\quad k^2>0 $ is simple, and
	\item  whenever $ f $ is not supported on a loop (if such exists),
	\[(\tr_{k}(f))_{j}\ne0,\quad\mbox{for all}\quad j=1,2,\ldots M.\] 
\end{enumerate}
We will prove Theorem \ref{thm: non vanishing} by applying Theorem \ref{thm: unified} to the polynomial $ q(\bz,\xv):=\prod_{j=1}^{m}x_{j} $ together with the following lemma.
\begin{lem}\label{lem: nonvanishing lemma}
	Let $ \Gamma $ be a graph satisfying Assumption \ref{ass: assumptions}, and fix an index $ j\in\{1,2,\ldots,m\} $. Then, there exist an $ \lv\in\R_{+}^{N} $ and an eigenpair $ (k^2,f) $ of $ (\Gamma,\lv) $, such that $ k^2 $ is a non-zero simple eigenvalue, $ f $ is not supported on a loop (if the graph has loops), and
	\[(\tr_{k}(f))_{j}\ne0.\]
\end{lem} 
\begin{proof}
	First consider the case where the $ j $-th coordinate of $ \tr_{k}(f) $ is a Dirichlet coordinate (namely $ A_{j'} $ or $ C_{j'} $ for some edge $ e_{j'} $) and therefore equal to the value of $ f $ at some vertex $ v $. In such case, the genericity result Berkolaiko and Liu in \cite{BerLiu_jmaa17} assures that there is an $ \lv\in\R_{+}^{N} $ (in fact a residual set of such) for which every eigenfunction which is not supported on a loop does not vanish at $ v $. Since not all eigenfunctions are supported on loops (as can be seen in the proof of Theorem \ref{thm: disjoint spectrum}) then we are done. \\

	Now consider the case of $ j$ such that the $ j $-th coordinate of $ \tr_{k}(f) $ is a Neumann coordinate, namely  it equals to $ \frac{1}{k} $ times the normal derivative of $ f $ along an edge $ e $ at a vertex $ v $ which is not of degree one. We may now use \cite[Lemma 5.20]{Alon2020PHD} which shows that there exists some $ \kv\in\R^{N}/2\pi\Z^{N} $ such that whenever $ \lv\in\R_{+}^{N} $ satisfies $ \exp(i\lv)=\exp(i\kv)$,  then $ k^2=1 $ is a simple eigenvalue of $ (\Gamma,\lv) $ with an eigenfunction $ f $ which is not supported on a loops (this is the meaning of the notation $ \kv\in\Sigma_{\mathcal{L}}^{c} $ in \cite{Alon2020PHD}), and furthermore the normal derivative of $ f $ along the edge $ e $ at the vertex $ v $ is non-zero (this is the meaning of the notation $ \partial_{e}f_{\kv}(v)\ne 0 $ in \cite{Alon2020PHD}). This proves the lemma.
\end{proof}
\begin{proof}[Proof of Theorem \ref{thm: non vanishing}]
	Let $ \Gamma $ be a graph with $ N $ edges that satisfies Assumption \ref{ass: assumptions}. Choose the irreducible polynomial $ p $ as follows. If $ \Gamma $ has loops, set $ p=P_{\Gammasym} $. If $ \Gamma $ is a mandarin, set $ p=P_{M,\mathrm{s}} $. Otherwise, $ \PG $ is irreducible and we set $ p=\PG $. According to Lemma \ref{lem: trace space for loops}, if $ \Gamma $ has loops then any eigenpair $ (k^2,f) $ of $ (\Gamma,\lv) $ with the properties that $ k^2 $ is non-zero and simple, and $ f $ is not supported on a loop, must satisfy
	\[\bz=\exp(ik\lv)\in Z(P_{\Gammasym})\cap\mreg(\Gamma)=Z(p)\cap\mreg(\Gamma).\]
	We may deduce from Lemma \ref{lem: nonvanishing lemma} that for any $ \Gamma $ which is not a mandarin 
	\begin{equation}\label{eq: existance eq non mandarin}
		\exists (\bz,\xv)\in\TS(\Gamma)\quad\mbox{such that}\quad \bz\in\mreg(\Gamma)\cap Z(p)\quad\mbox{and}\quad q(\xv)\ne0.
	\end{equation}
Notice that $ q(\xv)\ne0 $ implies $ \xv\ne0 $. This is the needed assumption for Theorem \ref{thm: unified} and we conclude that the properties 
\begin{enumerate}
	\item $\quad k^2>0 $ is simple, and
	\item $ q(\tr_{k}(f))_{j}\ne0 $ whenever $ f $ is not supported on a loop (if such exists),
\end{enumerate}  
are strongly and ergodically generic. We have proved Theorem \ref{thm: non vanishing} except for mandarin graphs.\\

Now assume that $ \Gamma $ is a mandarin graph, an orient all edges from one vertex, say $ v_{0} $ to the other, say $ v_{1} $. According to Lemma \ref{lem: nonvanishing lemma}, there exists an eigenpair $ (k^2,f) $ of $ (\Gamma,\lv) $ such that $ k^2 $ is a non-zero simple eigenvalue and $ q(\tr_{k}(f))=0 $. We claim that we may assume that $ f $ is symmetric. To see that, first assume that $ f $ is anti-symmetric, and consider its restriction to the edge $ e_{j} $
\[f|_{e_{j}}(t_{j})=A_{j}\cos(kt_{j})+B_{j}\cos(kt_{j}),\qquad t_{j}\in[0,\ell_{j}]\]
Define $ \lv'$ such that $ \ell'_{j}=\ell_{j}+\frac{\pi}{k} $ for all edges, and extend $ f $ to a function $ \tilde{f} $ on $ (\Gamma,\lv') $
\[\tilde{f}|_{e_{j}}(t_{j})=A_{j}\cos(kt_{j})+B_{j}\cos(kt_{j}),\qquad t_{j}\in[0,\ell_{j}+\frac{\pi}{k}].\]
It is not hard to conclude that $ \tr_{k}(\tilde{f}) $ and $ \tr_{k}(f) $ are related by 
\[(\tilde{A}_{j},\tilde{B}_{j},\tilde{C}_{j},\tilde{D}_{j})=(A_{j},B_{j},-C_{j},-D_{j}).\]
We may conclude that $ \tilde{f} $ is a symmetric eigenfunction of eigenvalue $ k^2 $ with $ q(\tr_{k}(\tilde{f}))\ne 0 $. To see that $ k^2 $ is also simple as an eigenvalue of $ (\Gamma,\lv') $, notice that this extension procedure is invertible and maps eigenspaces to eigenspaces, so it preserves multiplicity.\\

To conclude, we have shown that there exists an eigenpair $ (k^2,f) $ of $ (\Gamma,\lv) $ such that $ k^2 $ is a non-zero simple eigenvalue, $ q(\tr_{k}(f))=0 $, and $ f $ is symmetric, so \eqref{eq: existance eq non mandarin} holds and we may apply Theorem \ref{thm: unified} by which
\begin{enumerate}
	\item $\quad k^2>0 $ is simple, and
	\item $ q(\tr_{k}(f))\ne0 $ whenever $ f $ is symmetric,
\end{enumerate}  
are strongly and ergodically generic properties for a mandarin graph. If we set $ p=P_{M,\mathrm{as}} $, then the same argument proves that  
\begin{enumerate}
	\item[(3)] $ q(\tr_{k}(f))\ne0 $ whenever $ f $ is anti-symmetric,
\end{enumerate}  
is also strongly and ergodically generic. Since an eigenfunction of a simple eigenvalue is either symmetric or anti-symmetric, then we are done.
\end{proof}

\section{Future work}
\subsection{The conjecture of $ \Q $-independent spectrum}
In this paper we constructed a machinery for proving genericity statements for a single eigenpair. However, this machinery may be generalized to include relations between different eigenvalues and eigenfunctions, by considering products of $ \TS(\Gamma) $ and products of $ \Sigma(\Gamma) $. Kurasov and Sarnak has shown in \cite{KurSar2022} that when $ \lv $ is $ \Q $-independent, the spectrum $ \spec(\Gamma,\lv) $ has infinite dimension over $ \Q $. Following this result, Sarnak raised the question of whether $ \spec(\Gamma,\lv) $ is linearly independent over $ \Q $, for a generic $ \lv $. We believe that the answer is affirmative, and we write it as a conjecture.
\begin{conjecture}[$ \Q $-independent spectrum]
	For every graph $ \Gamma $, maybe except some pathological cases, there is a generic set $ G\subset\R_{+}^{N} $, such that for any $ \lv\in G $, the spectrum $ \spec(\Gamma,\lv) $ is linearly independent over $ \Q $. That is, if we number the non-zero square root eigenvalues in $ \spec(\Gamma,\lv) $ increasingly, $ k_{1}\le k_{2}\le k_{3}\le\ldots\nearrow \infty $. Then, for any $ n\in\N $, 
	\[\sum_{j=1}^{n}k_{j}q_{j}\ne0,\quad\mbox{for all}\quad \textbf{q}=(q_{1},\ldots,q_{n})\in\Q^{n}\setminus\{0\}. \]   
\end{conjecture}
If we can extend Theorem \ref{thm: unified} to products of the trace space, then we would reduce the conjecture, for a given graph $ \Gamma $, to the following problem. For any rational vector $ \textbf{q}\in\Q^{n}\setminus\{0\} $, provide at least one choice of $ \lv $, such that the first $ n $ eigenvalues are simple and $ \sum_{j=1}^{n}k_{j}q_{j}\ne0 $. 
\subsection{The conjecture of full support eigenfunctions}
We say that an eigenfunction has full support if it does not vanish entirely on any edge,
\[f|_{e_{j}}\not\equiv 0 ,\quad\mbox{for all}\quad j=1,2,\ldots,N.\]
Conjecture 4.3 in \cite{HofKenMugPlu2021pleijel} states,
\begin{conjecture}[Full support eigenfunctions]\cite{HofKenMugPlu2021pleijel}
	For any metric graph $ (\Gamma,\lv) $, and any choice of a complete orthonormal sequence of eigenfunctions, there are infinitely many eigenfunctions with full support.
\end{conjecture}
This conjecture was confirmed in \cite{HofKenMugPlu2021pleijel} for any graph with $ \lv $ proportional to a rational vector. The case of $ \Q $-independent $ \lv $ follows from the ergodic genericity proved in \cite{AloBanBer_cmp18}. The conjecture was proven for all tree graphs (and graphs with Dirichlet conditions) in \cite{PluTau2021fully}. Let us provide two lemmas that may lead to progress in proving this conjecture in general.
\begin{lem}\label{lem: full support1}
	Consider a metric graph $ (\Gamma,\lv) $. If there exists one non-zero eigenvalue which is simple and has eigenfunction with full support, then there are infinitely many such simple eigenvalues whose eigenfunctions have full support. 
\end{lem}
\begin{proof}
	Using Lemma \ref{lem: vanishihng on an edge}, we may say that a non-zero $ k\in\spec(\Gamma,\lv) $ is simple and has eigenfunction of full support, if 
	\begin{equation}\label{eq: non vanishihng derivatives}
		\frac{\partial}{\partial z_{j}}\PG(\exp(ik\lv))\ne0\quad\mbox{for every}\quad j=1,2,\ldots,N. 
	\end{equation}
	Given such $ k $, consider the infinite path, $ \exp(it\lv) $ for $ t\in\R_{+} $. This path intersects any $ \Sigma(\Gamma) $ neighborhood of the point $ \exp(ik\lv) $ infinitely often. Taking a small enough neighborhood so that the derivatives of $ \PG $ remains non-zero, we get an infinite sub-sequence of $ \spec(\Gamma,\lv) $ of square root eigenvalues that satisfy \eqref{eq: non vanishihng derivatives}, and hence each of these square root eigenvalues is simple, with eigenfunctions that has full support.
\end{proof}
Lemma \ref{lem: full support1}, implicitly, is used in \cite{PluTau2021fully}, where the fact that the graph is a tree allows to omit the $ k\ne0 $ restriction, and then one can take $ k=0 $, which is simple and has a constant eigenfunction (and hence of full support). The next lemma uses the trace space to capture the property of ``not having full support" in terms of the secular manifold.
\begin{lem}
	Given a metric graph $ (\Gamma,\lv)$, let $ f $ be an eigenfunction with a non zero eigenvalue $ k^2\ne 0 $, and let $ \xv=\tr_{k}(f) $. Let $ \mathrm{supp}(f)\subset\E $ be the set of edges $ e_{j} $ for which $ f|_{e_{j}}\not\equiv 0 $ and let $ s_{j}:=e^{ik\ell_{j}} $ for every $ e_{j}\in\mathrm{supp}(f)$. Then, the $ \bz $ fiber 
	\[\TS(\Gamma)_{\xv}:=\set{\bz\in\Sigma(\Gamma)}{(\bz,\xv)\in\TS(\Gamma)},\]
	is an $ m=N-|\mathrm{supp}(f)| $ dimensional torus inside $ \Sigma(\Gamma) $, given by
		\[\TS(\Gamma)_{\xv}=\set{\bz\in\T^{N}}{z_{j}=s_{j}\quad\mbox{for every }\quad e_{j}\in \mathrm{supp}(f)}.\]
\end{lem}
 \begin{proof}
 	Let $( A_{j},B_{j},C_{j},D_{j})$ be the restriction of $ \xv $ to the edge $ e_{j} $. Since $ \xv $ satisfies the vertex conditions, then according to Lemma \ref{lem: trace space polynomial equation}, $ \bz\in\TS(\Gamma)_{\xv} $ if and only if
 		\begin{align}
 		 A_{j}+iB_{j}- z_{j}(C_{j}-iD_{j}) & =0\\
 		 C_{j}+iD_{j}- z_{j}(A_{j}-iB_{j}) & =0,
 	\end{align}
for every edge $ e_{j} $. If $ e_{j}\notin\mathrm{supp}(f) $, namely $( A_{j},B_{j},C_{j},D_{j})=0  $, then any $ z_{j} $ solves these equations. If $ e_{j}\in\mathrm{supp}(f) $, namely $( A_{j},B_{j},C_{j},D_{j})\ne0  $, then these two equations have a unique $ z_{j} $ solution. Since we are given a point $ \exp(ik\lv)\in\TS(\Gamma)_{\xv} $, then the $ z_{j} $ solution for $ e_{j}\in\mathrm{supp}(f) $ is 
\[z_{j}=e^{ik\ell_{j}}=:s_{j}.\]
 \end{proof}
 Using the above two lemmas, one may prove the conjecture by showing that there is no $ \lv\in\R_{+}^{N} $ for which the path $ t\mapsto\exp(it\lv) $ intersects $ \Sigma(\Gamma) $ only at positive dimensional subtori as above. It is possible that such a claim can be approached using algebraic tools.
\subsection{The co-dimension of the singular set $ \msing(\Gamma) $}
In \cite{CdV_ahp15}, below the proof of Proposition 1.1, the author conjecture that in the cases where $ \Sigma(\Gamma) $ is irreducible, the singular set $ \msing(\Gamma) $ has real dimension
\[\dim(\msing(\Gamma))\le N-3.\]
Consider the vertical fiber,
\[\TS(\Gamma)_{\bz}:=\set{\xv\in\C^{4N}}{(\bz,\xv)\in\TS(\Gamma)},\]
which is a complex vector space. Recall that 
\[\bz\in\msing(\Gamma)\iff \dim(\TS(\Gamma)_{\bz})\ge 2.\]
Since $ \TS(\Gamma) $ is defined in $ \T^{N}\times\C^{4N} $ by $ 4N $ linear equations with real coefficients, then we expect $ \TS(\Gamma) $ to be an $ N $ dimensional manifold. However this requires a transversality argument, i.e., showing that the rank of the Jacobian is always maximal. If this is the case, then the submanifold  
\[\Sigma_{d}(\Gamma):=\set{\bz\in\Sigma(\Gamma)}{\dim(\TS(\Gamma)_{\bz})=d}, \]
should have dimension at most $ N-d $. If this is true, then to prove the conjecture, one only needs to deal with $ \Sigma_{2}(\Gamma) $, namely eigenvalues with multiplicity exactly $ 2 $. One may ask if such multiplicity can exist without any symmetry of degree $ 2 $. We believe that this approach might lead to a proof for this conjecture.    
	\appendix
	\section{Algebraic varieties intersected with the torus}
	\begin{lem}\label{lem: real dimension bounded by complex dimension}
		Consider an algebraic set (or variety) $ V\subset\C^{N} $ of (complex) dimension $ n $, then $ V\cap\T^{N} $ has real dimension at most $ n $.
	\end{lem}
	\begin{proof}
		We consider the case where $ V $ is the common zero set of $ p_{1},p_{2},\ldots,p_{m} $ are distinct irreducible polynomials, which we write as 
		\[V=Z(p_{1})\cup Z(p_{2})\ldots\cup Z(p_{m}).\]
		Since any variety is a finite union of zero sets as above, then it is sufficient to prove the statement for this case.\\
		
		Let $ Q=(p_{1},p_{2},\ldots,p_{m}) $ so that $ Q:\C^{N}\to \C^{m} $ and let $ F(\xv):=Q(\exp(i\xv)) $. We denote their derivatives by $ DF $ and $ DQ $. The derivatives at a point $ \bz=\exp(i\xv)\in\T^{N} $ are $ m\times N $ matrices, which are related by the diagonal unitary matrix $u(\bz):= i\diag(\bz) $,
		\begin{equation}
			DF(\xv)=DQ(\bz)u(\bz),\qquad \bz=\exp(i\xv)\in\T^{N}.
		\end{equation} 
		Assume that $  V\cap\T^{N} $ has real dimension $ n' $, so that we want to show $ n'\le n $. Since the exponent is a diffeomorhpism between $ \R^{N}/2\pi\Z^{N} $ and $ \T^{N} $, then
		\[O:=\set{\xv\in\R^{N}/2\pi\Z^{N}}{\exp(i\xv)\in V\cap\T^{N}},\] has real dimension $ n' $ and $ \exp(i\xv) $ is a regular point of $  V\cap\T^{N} $ if and only if $ \xv $ is a regular point of $ O $. Notice that $ O $ is the zero set of $ F $, and therefore at any regular point $ \xv\in O $, the tangent space $ T_{\xv}O $ is $ n' $ dimensional real vector space, given by the right kernel of the matrix $ DF(\xv) $. In particular, $ T_{\xv}O $ has an orthonormal basis of $ n' $ real vectors $\{\textbf{a}_{1},\textbf{a}_{2},\ldots,\textbf{a}_{n'}\} $, such that 
		\[DF(\xv)\textbf{a}_{j}=0,\qquad j=1,2,\ldots n'.\]
		Define the vectors $\textbf{b}_{j}:=u(\bz)\textbf{a}_{j}$ for all $ j $. Since $ u(\bz) $ is unitary, these are $ n' $ complex orthonormal vectors. They satisfy 
		\[DF(\xv)\textbf{a}_{j}=DQ(\bz)\textbf{b}_{j}=0,\qquad j=1,2,\ldots n',\]
		so we conclude that the kernel of $ DQ(\bz) $ has complex dimension at least $ n' $.\\
		Now, assume that there exists a regular point of $ V $ in $ V\cap\T^{N} $, say $ \bz=\exp(i\xv) $, then the tangent space $ T_{z}V $ at $ \bz $ is $ n $ dimensional and equals to the kernel of $ DQ(\bz) $. Hence,
		\[n'\le n.\]
		On the other hand if there are no regular points of $ V $ in $  V\cap\T^{N} $, then $  V\cap\T^{N} $ is contained in the singular set of $ V $, say $ V^{\mathrm{sing}} $, which is a variety of dimension strictly smaller than $ n $. Set $ V=V_{0} $ and $ n_{0}=n $ and let $ V_{j+1} $ be the singular set of $ V_{j} $ with $ n_{j+1}=\dim (V_{j+1}) $. We may conclude that for any $ j $, either $  V_{j}\cap\T^{N} $ has real dimension $ n'\le n_{j} $ or $  V_{j}\cap\T^{N}=  V_{j+1}\cap\T^{N}  $. Since the dimensions $ n_{j} $ are strictly decreasing this must end after at most $ n=n_{0} $ steps and provide the answer $ n'\le n=n_{0} $. 
\end{proof}
	\bibliographystyle{siam}
\bibliography{GlobalBib,bib}
\end{document}